\documentclass[preprint,12pt]{elsarticle}

\usepackage{amssymb}
\usepackage{amsthm}

\usepackage{amsmath}
\usepackage{bm}
\usepackage{booktabs}
\usepackage{mathrsfs}
\usepackage{xcolor}
\usepackage{changes}
\newtheorem{theorem}{Theorem}[section]
\newtheorem{lemma}{Lemma}[section]

\newtheorem{example}{Example}[section]
\newtheorem{remark}{Remark}[section]

\date{}
\journal{}

\begin{document}
	
	\begin{frontmatter}
		
		
		
		\title{Locking analysis and high-order locking-free finite element method for three-dimensional electroporoelasticity equations}
		
		
		\author[1]{Xuan Liu}
		\affiliation[1]{organization={School of Mathematics, Jilin University},
			city={Changchun},
			postcode={130012},
			country={China}}
		\ead{liuxuan2123@163.com}
		\author[1]{Yongkui Zou\corref{corresponding}}
		\cortext[corresponding]{Corresponding author.}
		\ead{zouyk@jlu.edu.cn}
		\author[2]{Yanzhao Cao}
		\affiliation[2]{organization={Department of Mathematics and Statistics, Auburn University},
			city={Auburn},
			postcode={AL 36849},
			country={USA}}
		\ead{yzc0009@auburn.edu}

		\begin{abstract}
			Electroporoelasticity equations couple Maxwell's equations with Biot's poroelasticity model and admit severe Poisson locking in standard conforming finite element discretizations when the Lam\'{e} constant is large. In this paper, we provide a rigorous analysis of the Poisson locking phenomenon for three-dimensional quasi-static electroporoelasticity equations and show that the spatial convergence order of conforming finite element approximations is reduced in the nearly incompressible regime. To eliminate this locking effect, we introduce a five-field formulation and a fully discrete high-order finite element method. We prove uniform stability with respect to large Lam\'{e} constant. Based on this, we establish a locking-free scheme by deriving its uniform error estimates with respect to the Lam\'{e} coefficient. The analysis covers the fully coupled electromagnetic–poroelastic system and applies to high-order elements in three dimensions. Extensive numerical experiments are presented to verify the theoretical convergence rates and demonstrate robustness with respect to the Lam\'{e} parameter.
		\end{abstract}

		
	\end{frontmatter}
	
	
\section{Introduction}\label{s7.1}
Electroporoelasticity equations describe the coupled interaction of electromagnetic fields, elastic deformation of a solid matrix, and fluid flow in porous media. They arise from the coupling of Maxwell's equations with Biot's poroelasticity model and have been used to model seismoelectric and electroseismic phenomena in fluid-saturated materials 
\cite{KuMuRi2006,PrGa2005,PrMoGa2004,WhZh2006}. These equations form a strongly coupled multiphysics system that combines features of electromagnetics, elasticity, and poroelasticity, leading to substantial analytical and computational challenges, particularly in three dimensions.

In this paper, we aim to construct a class of locking-free finite element methods (FEMs) for the electroporoelasticity equations. FEMs have been widely employed for numerical solutions of elasticity and poroelasticity problems, including conforming methods, mixed formulations, nonconforming methods, and discontinuous Galerkin approaches \cite{ChSu1997,BoBrFo2013,Zh1997,ZhChPa2023}.  For electroporoelasticity equations, numerical methods have been studied in recent years, primarily focusing on well-posedness, stability, and basic finite element discretizations \cite{HuMe2022,HuMe2023,LiZoZhCaMe2025}. It is well known that standard conforming finite element discretizations for elasticity and poroelasticity may suffer from locking phenomena, i.e., numerical solutions lose their accuracy when some of the model parameters reach to certain limits. In the case of electroporoelasticity, the so called Poisson locking may happen when the Lam\'{e}  constant becomes large, corresponding to the nearly incompressible regime \cite{LiZoZhCaMe2025-2}. In such cases, numerical solutions lose accuracy and convergence rates deteriorate unless the mesh is excessively refined.  

Many effective locking-free methods have been developed to overcome the locking phenomena for various partial differential equations, such as the $p$-version FEMs \cite{Vo1983}, $hp$-version FEMs \cite{StSu1996}, mixed FEMs \cite{ArBrDo1984,ArDoGu1984}, nonconforming FEMs \cite{BrSu1992,ChJiCu2005}, discontinuous Galerkin methods \cite{Wi2004,Wi2006}, weak Galerkin FEMs \cite{ZhZhCaMe2020,ZhWaZo2023}, enriched Galerkin methods \cite{LeYi2023,YiLeZi2022} and some other multi-physics field methods \cite{OyRu2016,Yi2017}. For electroporoelasticity, we \cite{LiZoZhCaMe2025-2} derived a first-order locking-free finite element scheme based on N\'{e}d\'{e}lec, nonconforming Crouzeix-Raviart and Lagrange elements, and  proved that the numerical solution tends to a divergence-free state as  the Lam\'{e}  constant goes to infinity, inheriting the asymptotic behavior of the exact solution and ensuring the locking-free property of the numerical method.

The primary contribution of this paper is the development of a high-order, provably locking-free finite element method for three-dimensional quasi-static electromagnetic–poroelastic systems. To overcome Poisson locking, we introduce a novel five-field formulation by defining an auxiliary variable that couples the pressure and volumetric strain. This reformulation leads to a system whose solution remains uniformly bounded for arbitrarily large Lam\'{e} constants. Based on this framework, we design a fully discrete high-order scheme, employing the backward Euler method for time discretization and a compatible combination of N\'{e}d\'{e}lec, Lagrange, and Taylor–Hood finite elements in space. We rigorously establish uniform stability and derive Lamé-constant-independent high-order error estimates, thereby providing a definitive proof that the proposed method is locking-free.

In addition, we establish a rigorous mathematical characterization of Poisson locking in conforming finite element discretizations of the three-dimensional quasi-static electromagnetic–poroelasticity equations. We prove that as the Lam\'{e} constant increases, the numerical solution loses uniform boundedness, and the spatial convergence rate deteriorates. This analysis provides a precise and quantitative explanation of the Poisson locking mechanism in this coupled electromagnetic–poroelastic system.

The rest of this paper is organized as follows. In Section \ref{s7.2}, we introduce quasi-static electroporoelasticity equations. In Section \ref{s7.3}, we investigate the Poisson locking phenomenon in conforming finite element approximations. In Section \ref{s7.4}, we introduce five-field electroporoelasticity equations and prove the uniform boundedness of solutions with respect to the Lam\'{e} constant. Then, we construct a high-order finite element scheme and study its well-posedness and locking-free error estimates. In Section \ref{s7.5}, we present numerical experiments to validate the theoretical results.

\section{Preliminaries}\label{s7.2}
In this section, we introduce quasi-static electroporoelasticity equations and some assumptions.

Let $[0,T]$ with $T>0$ be an interval and $\mathscr{D}\subset \mathbb{R}^{3}$ be an open bounded convex polyhedron with $\mathbf{n}$ the unit outward normal vector to its boundary $\partial\mathscr{D}$. The space $L^{2}(\mathscr{D})$ consists of square-integrable functions endowed with inner product $(\cdot,\cdot)$ and norm $\|\cdot\|$. Throughout this paper, the spaces of three-dimensional vector fields will be denoted in bold italic letters, e.g., $\bm{L}^{2}(\mathscr{D})=(L^{2}(\mathscr{D}))^{3}$ and its norm is still denoted by $\|\cdot\|$. For any integer $\ell\geq 1$, denote by $H^{\ell}:=H^{\ell}(\mathscr{D})$ the standard Sobolev spaces with norm $\|\cdot\|_{\ell}$ and $H^{\ell}_{0}=\{v\in H^{\ell}: D^{\iota}v|_{\partial\mathscr{D}}=0, \, |\iota|<\ell\}$ \cite{CaHoLi2020}. Let $C(\bar{\mathscr{D}})$ be a space consisting of continuous functions on $\bar{\mathscr{D}}$.
Define three subspaces of $\bm{L}^{2}(\mathscr{D})$
\begin{align*}
	\bm{H}(div) &= \{ \mathbf{v}: \mathbf{v}\in \bm{L}^{2}(\mathscr{D}), \ \nabla\cdot \mathbf{v}\in L^{2}(\mathscr{D}) \},
	\\
	\bm{H}(\mathbf{curl}) &= \{ \mathbf{v}: \mathbf{v}\in \bm{L}^{2}(\mathscr{D}), \ \nabla\times \mathbf{v}\in \bm{L}^{2}(\mathscr{D}) \},
	\\
	\bm{H}_{0}(\mathbf{curl}) &= \{ \mathbf{v}: \mathbf{v}\in \bm{H}(\mathbf{curl}), \ \mathbf{v}\times \mathbf{n}=\mathbf{0} \ \text{on} \ \partial\mathscr{D} \}.
\end{align*}

Consider three-dimensional quasi-static electroporoelasticity equations for $(t,\mathbf{x})\in (0,T]\times \mathscr{D}$,
\begin{equation}\label{e7.2.1}
	\begin{aligned}
		& \epsilon\frac{\partial}{\partial t}\mathbf{E} + \sigma \mathbf{E} - \nabla\times \mathbf{H} - L\nabla p = \mathbf{j},
		\\
		& \mu\frac{\partial}{\partial t}\mathbf{H} + \nabla\times \mathbf{E} = \mathbf{0},
		\\
		& -(\lambda+G)\nabla(\nabla\cdot \mathbf{u}) - G\Delta \mathbf{u} + \alpha\nabla p = \mathbf{f},
		\\
		& \frac{\partial}{\partial t}(c_{0}p + \alpha\nabla\cdot \mathbf{u}) - \kappa\Delta p + L\nabla\cdot \mathbf{E} = g.
	\end{aligned}
\end{equation}
The initial value conditions are, for $\mathbf{x}\in\mathscr{D}$,
\begin{equation}\label{e7.2.2}
	\mathbf{E}(0,\mathbf{x}) = \mathbf{E}_{0}(\mathbf{x}), \  \mathbf{H}(0,\mathbf{x}) = \mathbf{H}_{0}(\mathbf{x}),
	\  \mathbf{u}(0,\mathbf{x}) = \mathbf{u}_{0}(\mathbf{x}), \  p(0,\mathbf{x}) = p_{0}(\mathbf{x}),
\end{equation}
with $\mathbf{u}_{0}$ and $p_{0}$ fulfilling the third equation of \eqref{e7.2.1}. 
The boundary value conditions are, for $(t,\mathbf{x})\in (0,T]\times\partial \mathscr{D}$,
\begin{equation}\label{e7.2.3}
	\mathbf{E}(t,\mathbf{x})\times \mathbf{n} = \mathbf{0},
	\quad
	\mathbf{u}(t,\mathbf{x})=\mathbf{0}, 
	\quad 
	p(t,\mathbf{x}) = 0.
\end{equation}
Here, $\mathbf{E}$ is the electric field, $\mathbf{H}$ is the magnetic field, $\mathbf{u}$ is the displacement of solid matrix and $p$ is the pressure in the fluid. The physical parameters $\epsilon$, $\sigma$, $L$, $\mu$, $G$, $\alpha$, $c_{0}$ and $\kappa$ are assumed to be fixed positive constants and $\lambda\in [1,\infty)$ is the Lam\'{e} constant, cf.\ \cite{HuMe2022,HuMe2023,LiZoMe2026}. In this paper, we assume that $\mathbf{f}(t,\mathbf{x})\equiv\mathbf{0}$, otherwise, we can first search a solution $\tilde{\mathbf{u}}(t)$ of a stationary elasticity problem for any $t\in(0,T]$,
\begin{equation*}
	-(\lambda+G)\nabla(\nabla\cdot \tilde{\mathbf{u}})
	- G\Delta \tilde{\mathbf{u}}
	= \mathbf{f}.
\end{equation*}
Then, replace $\mathbf{u}$ by $\mathbf{u}+\tilde{\mathbf{u}}$ in \eqref{e7.2.1} to set up the homogeneous equation. For convenience, we also denote $\mathbf{E}(t,\mathbf{x})$ by either $\mathbf{E}(t)$ or $\mathbf{E}$, etc.

Define a family of symmetric bilinear forms
\begin{equation*}
	a_{\lambda}(\mathbf{w},\mathbf{v}) = ((\lambda+G)\nabla\cdot \mathbf{w}, \nabla\cdot \mathbf{v}) + (G\nabla \mathbf{w}, \nabla \mathbf{v}), \quad \forall \, \mathbf{w},\mathbf{v}\in \bm{H}^{1}_{0}.
\end{equation*}
Then, it is easy to verify that 
\begin{equation}\label{e7.2.4}
	\begin{aligned}
		a_{\lambda}(\mathbf{v},\mathbf{v}) \geq&\ C\|\mathbf{v}\|_{1}^{2},
		\\
		|a_{\lambda}(\mathbf{w},\mathbf{v})| 
		\leq&\ \lambda\|\mathbf{w}\|_{1}\|\mathbf{v}\|_{1} +  C\|\mathbf{w}\|_{1}\|\mathbf{v}\|_{1},
	\end{aligned}
\end{equation}
where $C>0$ is a constant independent of $\lambda$.

The variational equation of \eqref{e7.2.1}--\eqref{e7.2.3} is to seek $(\mathbf{E}(t),\mathbf{H}(t),\mathbf{u}(t),p(t))\in \bm{H}_{0}(\mathbf{curl}) \times \bm{L}^{2}(\mathscr{D}) \times \bm{H}^{1}_{0} \times H^{1}_{0}$ satisfying \eqref{e7.2.2} such that
\begin{equation}\label{e7.2.5}
	\begin{aligned}
		& \big(\epsilon\frac{\partial}{\partial t}\mathbf{E}(t), \mathbf{D}\big) 
		+ \big(\sigma \mathbf{E}(t), \mathbf{D}\big) 
		- \big(\mathbf{H}(t), \nabla\times \mathbf{D}\big) 
		- \big(L\nabla p(t), \mathbf{D}\big) 
		= \big(\mathbf{j}(t), \mathbf{D}\big),
		\\
		& \big(\mu\frac{\partial}{\partial t}\mathbf{H}(t), \mathbf{B}\big) 
		+ \big(\nabla\times \mathbf{E}(t), \mathbf{B}\big) = 0,
		\\
		& a_{\lambda}\big(\mathbf{u}(t), \mathbf{v}\big)
		- \big(p(t), \alpha \nabla\cdot \mathbf{v}\big) = 0,
		\\
		& \big(\frac{\partial}{\partial t}(c_{0}p(t)+\alpha\nabla\cdot \mathbf{u}(t)), q\big) 
		+ \big(\kappa\nabla p(t), \nabla q\big) 
		- \big(L\mathbf{E}(t), \nabla q\big) 
		= \big(g(t), q\big),
	\end{aligned}
\end{equation}
for every $t\in (0,T]$ and all $(\mathbf{D},\mathbf{B},\mathbf{v},q)\in \bm{H}_{0}(\mathbf{curl}) \times \bm{L}^{2}(\mathscr{D}) \times \bm{H}^{1}_{0} \times H^{1}_{0}$.

We assume that
\begin{itemize}
	\item [\bf{(H1)}] $0< L< \sqrt{\sigma\kappa}$.
	\item [\bf{(H2)}] The functions $\mathbf{E}_{0}$, $\mathbf{H}_{0}$, $\mathbf{u}_{0}$, $p_{0}$, $\mathbf{j}$ and $g$ are sufficiently smooth.
\end{itemize}

By \cite{LiZoZhCaMe2025}, \eqref{e7.2.5} has a unique solution $(\mathbf{E}(t),\mathbf{H}(t),\mathbf{u}(t),p(t))\in \bm{H}_{0}(\mathbf{curl}) \times \bm{L}^{2}(\mathscr{D}) \times \bm{H}^{1}_{0} \times H^{1}_{0}$ for every $\lambda\in [1,\infty)$. By virtue of \cite[Lemma 3.1]{LiZoZhCaMe2025-2}, there exists a constant $C>0$ independent of $\lambda$ such that for all $t\in (0,T]$,
\begin{equation*}
	\epsilon\|\mathbf{E}(t)\|^{2} + \mu\|\mathbf{H}(t)\|^{2} + c_{0}\|p(t)\|^{2} + (\lambda+G)\|\nabla\cdot \mathbf{u}(t)\|^{2} + G\|\nabla \mathbf{u}(t)\|^{2} \leq C.
\end{equation*}
From the Poincar\'{e} inequality, it follows that $\|\mathbf{u}(t)\|_{1} \leq C\|\nabla \mathbf{u}(t)\| \leq C$. Hence, the solution to \eqref{e7.2.5} is uniformly bounded with respect to $\lambda$.

We further assume that the exact solution to \eqref{e7.2.5} has additional regularity which is used to study convergence order for numerical approximations.

\begin{itemize}
	\item [\bf{(H3)}] For any integer $\beta\geq 1$, there exists a constant $C>0$ independent of $\lambda$ such that for all $t\in (0,T]$,
	\begin{align*}
		& \|\mathbf{E}(t)\|_{\beta+1}^{2} + \int_{0}^{t}{\|\frac{\partial}{\partial t} \mathbf{E}(\theta)\|_{\beta+1}^{2}}{\, \mathrm{d}\theta} 
		+ \int_{0}^{t}{\|\frac{\partial^{2}}{\partial t^{2}} \mathbf{E}(\theta)\|^{2}}{\, \mathrm{d}\theta} \leq C,
		\\
		& \|\mathbf{H}(t)\|_{\beta}^{2} + \int_{0}^{t}{\|\frac{\partial}{\partial t} \mathbf{H}(\theta)\|_{\beta}^{2}}{\, \mathrm{d}\theta} 
		+ \int_{0}^{t}{\|\frac{\partial^{2}}{\partial t^{2}} \mathbf{H}(\theta)\|^{2}}{\, \mathrm{d}\theta} \leq C,
		\\
		& \|\mathbf{u}(t)\|_{\beta+1}^{2} + \int_{0}^{t}{\|\frac{\partial}{\partial t} \mathbf{u}(\theta)\|_{\beta+1}^{2}}{\, \mathrm{d}\theta} 
		+ \int_{0}^{t}{\|\frac{\partial^{2}}{\partial t^{2}} \mathbf{u}(\theta)\|_{1}^{2}}{\, \mathrm{d}\theta} \leq C,
		\\
		& \|p(t)\|_{\beta+1}^{2} + \int_{0}^{t}{\|\frac{\partial}{\partial t} p(\theta)\|_{\beta+1}^{2}}{\, \mathrm{d}\theta} 
		+ \int_{0}^{t}{\|\frac{\partial^{2}}{\partial t^{2}} p(\theta)\|^{2}}{\, \mathrm{d}\theta} \leq C.
	\end{align*}
\end{itemize}

\section{Locking effect in conforming finite element approximation}\label{s7.3}
In this section, we first introduce a conforming FEM and then construct a fully discretized scheme to \eqref{e7.2.5} together with the backward Euler method. Then, we prove its well-posedness and derive $\lambda$-dependent error estimates. This explains why the Poisson locking phenomenon occurs in conforming finite element approximations to electroporoelasticity equations.

\subsection{Conforming FEM and full discretization}\label{s7.3.1}
Let $\mathcal{T}_{h}$ be a shape regular \cite[A1--A4]{WaYe2014} tetrahedral partition of $\mathscr{D}$. For each element $K\in \mathcal{T}_{h}$ with faces $\mathcal{F}_{K}$ and edges $\mathcal{E}_{K}$, let $h_{K}$ be its diameter and $h=\max_{K\in \mathcal{T}_{h}}h_{K}$ be the mesh size of $\mathcal{T}_{h}$. For any integer $l\geq 0$, denote by $P_{l}(K)$ and $\tilde{P}_{l}(K)$ the spaces of polynomials with degree no more than $l$ and homogeneous polynomials with degree of $l$, respectively.
For any $l\geq 1$, define two spaces consisting of vector-valued polynomials
\begin{align*}
	\bm{S}_{l}(K) =&\ \{\mathbf{L}\in \tilde{\bm{P}}_{l}(K): \mathbf{L}(\mathbf{x})\cdot\mathbf{x}=0, \ \forall\, \mathbf{x}\in K\},
	\\
	\bm{R}_{l}(K) =&\ \bm{P}_{l-1}(K)\oplus \bm{S}_{l}(K).
\end{align*}
Construct four finite element spaces
\begin{align*}
	\mathbb{E}_{h} =&\ \{\mathbf{E}_{h}\in \bm{H}_{0}(\mathbf{curl}): \mathbf{E}_{h}|_{K}\in \bm{R}_{l}(K), \ \forall\, K\in \mathcal{T}_{h}\},
	\\
	\mathbb{H}_{h} =&\ \{\mathbf{H}_{h}\in \bm{L}^{2}(\mathscr{D}): \mathbf{H}_{h}|_{K}\in \bm{P}_{l-1}(K), \ \forall\, K\in \mathcal{T}_{h}\},
	\\
	\mathbb{U}_{h} =&\ \{\mathbf{u}_{h}\in \bm{C}(\bar{\mathscr{D}}): \mathbf{u}_{h}|_{\partial\mathscr{D}}=0,\, \mathbf{u}_{h}|_{K}\in \bm{P}_{l}(K),\ \forall\, K\in \mathcal{T}_{h}\},
	\\
	\mathbb{P}_{h} =&\ \{p_{h}\in C(\bar{\mathscr{D}}): p_{h}|_{\partial\mathscr{D}}=0,\, p_{h}|_{K}\in P_{l}(K), \ \forall\, K\in \mathcal{T}_{h}\}.
\end{align*}
Notice that $\mathbb{E}_{h}$ is the space consisting of N\'{e}d\'{e}lec elements \cite{Ne1980}, and $\mathbb{H}_{h}$, $\mathbb{U}_{h}$ and $\mathbb{P}_{h}$ are the spaces comprising of Lagrange elements. 

Denote by $\mathfrak{P}_{h}: \bm{L}^{2}(\mathscr{D})\to \mathbb{E}_{h}$, $\mathcal{P}_{h}: \bm{L}^{2}(\mathscr{D})\to \mathbb{H}_{h}$, $\mathscr{P}_{h}: \bm{L}^{2}(\mathscr{D})\to \mathbb{U}_{h}$ and $P_{h}: L^{2}(\mathscr{D})\to \mathbb{P}_{h}$ the $L^{2}$-orthogonal projection operators, respectively. 
In the next lemma, we collect some properties of these operators.

\begin{lemma}\label{l7.3.1}
	(\cite[Lemma 11.18]{ErGu2021}) For any integer $\beta\geq 1$, there exists a constant $C>0$ independent of $h$ such that for any $\mathbf{E}\in \bm{H}^{\beta+1}$, $\mathbf{H}\in \bm{H}^{\beta}$, $\mathbf{u}\in \bm{H}^{\beta+1}$ and $p\in H^{\beta+1}$,
	\begin{align*}
		\|\mathfrak{P}_{h}\mathbf{E}-\mathbf{E}\| \leq&\ Ch^{\min\{l,\beta\}+1}\|\mathbf{E}\|_{\beta+1},
		\\
		\|\mathcal{P}_{h}\mathbf{H}-\mathbf{H}\| \leq&\ Ch^{\min\{l,\beta\}}\|\mathbf{H}\|_{\beta},
		\\
		\|\mathscr{P}_{h}\mathbf{u}-\mathbf{u}\| + h \|\mathscr{P}_{h}\mathbf{u}-\mathbf{u}\|_{1} 
		\leq&\ Ch^{\min\{l,\beta\}+1}\|\mathbf{u}\|_{\beta+1},
		\\
		\|P_{h}p-p\| + h \|P_{h}p-p\|_{1} \leq&\ Ch^{\min\{l,\beta\}+1}\|p\|_{\beta+1}.
	\end{align*}
\end{lemma}

Let $\mathcal{I}_{h}: \bm{H}^{2}\to \mathbb{E}_{h}$ be an interpolation operator such that $(\mathcal{I}_{h}\mathbf{E})|_{K}$ is the unique polynomial in $\bm{R}_{l}(K)$ having the same moments as $\mathbf{E}|_{K}$, cf.\ \cite{Ne1980,GiRa1986,Mo1991}. 
Let $\mathcal{Q}_{h}:\bm{H}^{1}_{0}\to \mathbb{U}_{h}$ and $\mathcal{R}_{h}:H^{1}_{0}\to \mathbb{P}_{h}$ be a pair of Ritz projection operators satisfying for any $\mathbf{u}\in \bm{H}^{1}_{0}$ and $p\in H^{1}_{0}$,
\begin{align}
	\label{e7.3.1}
	a_{\lambda}(\mathcal{Q}_{h}\mathbf{u},\mathbf{v}_{h}) - (\mathcal{R}_{h}p,\alpha\nabla\cdot \mathbf{v}_{h}) 
	=&\ a_{\lambda}(\mathbf{u},\mathbf{v}_{h}) - (p,\alpha\nabla\cdot \mathbf{v}_{h}), \quad \forall \, \mathbf{v}_{h}\in \mathbb{U}_{h}, 
	\\
	\label{e7.3.2}
	(\kappa\nabla\mathcal{R}_{h}p,\nabla q_{h}) =&\ (\kappa \nabla p,\nabla q_{h}), \quad \forall \, q_{h}\in \mathbb{P}_{h}.
\end{align}

\begin{lemma}\label{l7.3.2}
	For any integer $\beta\geq 1$, there exists a constant $C>0$ independent of $h$ and $\lambda$ such that for any $\mathbf{E}\in \bm{H}^{\beta+1}$, $\mathbf{u}\in \bm{H}^{\beta+1}\cap \bm{H}^{1}_{0}$ and $p\in H^{\beta+1}\cap H^{1}_{0}$,
	\begin{align*}
		(i)\ & \|\mathcal{I}_{h}\mathbf{E}-\mathbf{E}\| + \|\nabla\times (\mathcal{I}_{h}\mathbf{E}-\mathbf{E})\| 
		\leq Ch^{\min\{l,\beta\}}\|\mathbf{E}\|_{\beta+1},
		\\
		(ii)\ & \|\mathcal{R}_{h}p-p\| + h \|\mathcal{R}_{h}p-p\|_{1} \leq Ch^{\min\{l,\beta\}+1}\|p\|_{\beta+1},
		\\
		(iii)\ & \|\mathcal{Q}_{h}\mathbf{u}-\mathbf{u}\|_{1} 
		\leq C\lambda h^{\min\{l,\beta\}}\|\mathbf{u}\|_{\beta+1} 
		+ C h^{\min\{l,\beta\}}\|\mathbf{u}\|_{\beta+1} 
		\\
		& \qquad\qquad\qquad + Ch^{\min\{l,\beta\}+1}\|p\|_{\beta+1},
		\\
		(iv)\ & (\lambda+G)\|\nabla\cdot(\mathcal{Q}_{h}\mathbf{u}-\mathscr{P}_{h}\mathbf{u})\|^{2} 
		\leq C\lambda^{2} h^{2\min\{l,\beta\}}\|\mathbf{u}\|_{\beta+1}^{2} 
		\\
		& \qquad\qquad\qquad + C h^{2\min\{l,\beta\}}\|\mathbf{u}\|_{\beta+1}^{2} 
		+ Ch^{2(\min\{l,\beta\}+1)}\|p\|_{\beta+1}^{2}.
	\end{align*}
\end{lemma}

\begin{proof}
	The assertions $(i)$ and $(ii)$ directly follow from \cite[Theorem 2]{Ne1980} and \cite[Lemma 1.1]{Th2006}, respectively. 
	According to \eqref{e7.3.1}, we have
	\begin{equation*}
		a_{\lambda}(\mathcal{Q}_{h}\mathbf{u}-\mathscr{P}_{h}\mathbf{u},\mathbf{v}_{h}) 
		= a_{\lambda}(\mathbf{u}-\mathscr{P}_{h}\mathbf{u},\mathbf{v}_{h}) + (\mathcal{R}_{h}p-p,\alpha\nabla\cdot \mathbf{v}_{h}), \quad \forall \, \mathbf{v}_{h}\in \mathbb{U}_{h}.
	\end{equation*}
	Taking $\mathbf{v}_{h} = \mathcal{Q}_{h}\mathbf{u}-\mathscr{P}_{h}\mathbf{u} \in \mathbb{U}_{h}$ in the above equation and using \eqref{e7.2.4}, we get
	\begin{equation}\label{e7.3.3}
		\begin{aligned}
			&\ C \|\mathcal{Q}_{h}\mathbf{u}-\mathscr{P}_{h}\mathbf{u}\|_{1}^{2} 
			\leq |a_{\lambda}(\mathcal{Q}_{h}\mathbf{u}-\mathscr{P}_{h}\mathbf{u},\mathcal{Q}_{h}\mathbf{u}-\mathscr{P}_{h}\mathbf{u})|
			\\
			\leq&\ |a_{\lambda}(\mathbf{u},-\mathscr{P}_{h}\mathbf{u},\mathcal{Q}_{h}\mathbf{u}-\mathscr{P}_{h}\mathbf{u})| 
			+ |(\mathcal{R}_{h}p-p,\alpha\nabla\cdot (\mathcal{Q}_{h}\mathbf{u}-\mathscr{P}_{h}\mathbf{u}))| 
			\\
			\leq&\ \lambda\|\mathbf{u}-\mathscr{P}_{h}\mathbf{u}\|_{1} \|\mathcal{Q}_{h}\mathbf{u}-\mathscr{P}_{h}\mathbf{u}\|_{1} 
			+ C\|\mathbf{u}-\mathscr{P}_{h}\mathbf{u}\|_{1} \|\mathcal{Q}_{h}\mathbf{u}-\mathscr{P}_{h}\mathbf{u}\|_{1} 
			\\
			&\ + \alpha\|\mathcal{R}_{h}p-p\| \|\mathcal{Q}_{h}\mathbf{u}-\mathscr{P}_{h}\mathbf{u}\|_{1},
		\end{aligned}
	\end{equation}
	which together with Lemma \ref{l7.3.1} and $(ii)$ leads to
	\begin{equation}\label{e7.3.4}
		\begin{aligned}
			&\ \|\mathcal{Q}_{h}\mathbf{u}-\mathscr{P}_{h}\mathbf{u}\|_{1} 
			\leq C\lambda\|\mathbf{u}-\mathscr{P}_{h}\mathbf{u}\|_{1} 
			+ C\|\mathbf{u}-\mathscr{P}_{h}\mathbf{u}\|_{1}  
			+ C\|\mathcal{R}_{h}p-p\| 
			\\
			\leq&\ C\lambda h^{\min\{l,\beta\}}\|\mathbf{u}\|_{\beta+1} 
			+ Ch^{\min\{l,\beta\}}\|\mathbf{u}\|_{\beta+1} 
			+ Ch^{\min\{l,\beta\}+1}\|p\|_{\beta+1}.
		\end{aligned}
	\end{equation}
	Hence, we obtain
	\begin{align*}
		\|\mathcal{Q}_{h}\mathbf{u}-\mathbf{u}\|_{1} \leq&\ \|\mathcal{Q}_{h}\mathbf{u}-\mathscr{P}_{h}\mathbf{u}\|_{1} + \|\mathscr{P}_{h}\mathbf{u}-\mathbf{u}\|_{1} 
		\\
		\leq&\ C\lambda h^{\min\{l,\beta\}}\|\mathbf{u}\|_{\beta+1} 
		+ Ch^{\min\{l,\beta\}}\|\mathbf{u}\|_{\beta+1} 
		+ Ch^{\min\{l,\beta\}+1}\|p\|_{\beta+1}.
	\end{align*}
	By virtue of \eqref{e7.3.3} and \eqref{e7.3.4}, we have
	\begin{align*}
		&\ (\lambda+G)\|\nabla\cdot(\mathcal{Q}_{h}\mathbf{u}-\mathscr{P}_{h}\mathbf{u})\|^{2} 
		\leq |a_{\lambda}(\mathcal{Q}_{h}\mathbf{u}-\mathscr{P}_{h}\mathbf{u},\mathcal{Q}_{h}\mathbf{u}-\mathscr{P}_{h}\mathbf{u})| 
		\\
		\leq&\ C\lambda^{2} h^{2\min\{l,\beta\}}\|\mathbf{u}\|_{\beta+1}^{2} 
		+ Ch^{2\min\{l,\beta\}}\|\mathbf{u}\|_{\beta+1}^{2} 
		+ Ch^{2(\min\{l,\beta\}+1)}\|p\|_{\beta+1}^{2},
	\end{align*}
	which completes the proof.
\end{proof}

For any integer $N>0$, let $\tau=T/N$ be a temporal step-size and $t_{n}=n\tau$ $(0\leq n\leq N)$ be the time partition nodes. We apply the backward Euler method and conforming FEM to construct a fully discretized scheme to \eqref{e7.2.5}: find $(\mathbf{E}_{h}^{n},\mathbf{H}_{h}^{n},\mathbf{u}_{h}^{n},p_{h}^{n})\in \mathbb{E}_{h}\times \mathbb{H}_{h}\times \mathbb{U}_{h}\times \mathbb{P}_{h}$ such that
\begin{equation}\label{e7.3.5}
	(\mathbf{E}_{h}^{0},\mathbf{H}_{h}^{0},\mathbf{u}_{h}^{0},p_{h}^{0})
	= (\mathfrak{P}_{h}\mathbf{E}_{0},\mathcal{P}_{h}\mathbf{H}_{0},\mathscr{P}_{h}\mathbf{u}_{0},P_{h}p_{0}),
\end{equation}
and
\begin{equation}\label{e7.3.6}
	\begin{aligned}
		& (\epsilon \bar\partial_{t}\mathbf{E}_{h}^{n},\mathbf{D}_{h}) 
		+ (\sigma \mathbf{E}_{h}^{n},\mathbf{D}_{h}) 
		- (\mathbf{H}_{h}^{n},\nabla\times \mathbf{D}_{h}) 
		- (L\nabla p_{h}^{n},\mathbf{D}_{h}) 
		= (\mathbf{j}(t_{n}),\mathbf{D}_{h}),
		\\
		& (\mu \bar\partial_{t}\mathbf{H}_{h}^{n},\mathbf{B}_{h}) 
		+ (\nabla\times \mathbf{E}_{h}^{n},\mathbf{B}_{h}) = 0,
		\\
		& a_{\lambda}(\mathbf{u}_{h}^{n}, \mathbf{v}_{h}) 
		- (p_{h}^{n}, \alpha\nabla\cdot\mathbf{v}_{h})
		= 0,
		\\
		& \big(\bar\partial_{t}(c_{0}p_{h}^{n}+\alpha\nabla\cdot\mathbf{u}_{h}^{n}), q_{h}\big)
		+ (\kappa\nabla p_{h}^{n},\nabla q_{h}) 
		- (L\mathbf{E}_{h}^{n},\nabla q_{h}) 
		= (g(t_{n}),q_{h}),
	\end{aligned}
\end{equation}
for all $n=1,2,\ldots,N$ and $(\mathbf{D}_{h},\mathbf{B}_{h},\mathbf{v}_{h},q_{h})\in \mathbb{E}_{h}\times \mathbb{H}_{h}\times \mathbb{U}_{h}\times \mathbb{P}_{h}$, 
where $\bar\partial_{t}\mathbf{E}_{h}^{n} = (\mathbf{E}_{h}^{n}-\mathbf{E}_{h}^{n-1})/ \tau$, etc. 
In the next theorem, we study the $\lambda$-dependent a priori estimates of the fully discretized approximation.

\begin{theorem}\label{t7.3.1}
	Assume that (H1) and (H2) hold. The fully discretized equations \eqref{e7.3.6} and \eqref{e7.3.5} admit a unique solution $(\mathbf{E}_{h}^{n},\mathbf{H}_{h}^{n},\mathbf{u}_{h}^{n},p_{h}^{n})$. Furthermore, there exists a constant $C>0$ independent of $\tau$, $h$ and $\lambda$ such that for all $1\leq n\leq N$,
	\begin{equation*}
		\|\mathbf{E}_{h}^{n}\|^{2} 
		+ \|\mathbf{H}_{h}^{n}\|^{2} 
		+ \|\mathbf{u}_{h}^{n}\|_{1}^{2} 
		+ \|p_{h}^{n}\|^{2} 
		\leq C\lambda+C.
	\end{equation*}
\end{theorem}

\begin{proof}
	Notice that the number of linear equations \eqref{e7.3.6} equals the number of unknowns. Therefore, the uniqueness of its solution implies the existence. Here, we only prove the uniqueness.
	
	Taking $(\mathbf{D}_{h},\mathbf{B}_{h},\mathbf{v}_{h},q_{h})=(\mathbf{E}_{h}^{n},\mathbf{H}_{h}^{n},\bar\partial_{t}\mathbf{u}_{h}^{n},p_{h}^{n})$ in \eqref{e7.3.6}, and summing these up, we have
	\begin{equation}\label{e7.3.7}
		\begin{aligned}
			&\ (\epsilon \bar\partial_{t}\mathbf{E}_{h}^{n},\mathbf{E}_{h}^{n}) 
			+ (\mu \bar\partial_{t}\mathbf{H}_{h}^{n},\mathbf{H}_{h}^{n}) 
			+ ((\lambda+G)\nabla\cdot \mathbf{u}_{h}^{n}, \nabla\cdot\bar\partial_{t}\mathbf{u}_{h}^{n})
			\\
			&\ + (G\nabla\mathbf{u}_{h}^{n}, \nabla\bar\partial_{t}\mathbf{u}_{h}^{n}) 
			+ (c_{0}\bar\partial_{t} p_{h}^{n},p_{h}^{n}) 
			+ \sigma \|\mathbf{E}_{h}^{n}\|^{2} 
			+ \kappa \|\nabla p_{h}^{n}\|^{2} 
			\\
			= &\ 2(L \mathbf{E}_{h}^{n},\nabla p_{h}^{n}) 
			+ (\mathbf{j}(t_{n}),\mathbf{E}_{h}^{n}) 
			+ (g(t_{n}),p_{h}^{n}).
		\end{aligned}
	\end{equation}
	Direct computation leads to
	\begin{equation}\label{e7.3.8}
		\begin{aligned}
			\tau(\bar\partial_{t}\mathbf{E}_{h}^{n},\mathbf{E}_{h}^{n}) 
			=&\ \frac{\tau^{2}}{2}\|\bar\partial_{t}\mathbf{E}_{h}^{n}\|^{2} 
			+ \frac{1}{2}\|\mathbf{E}_{h}^{n}\|^{2} 
			- \frac{1}{2}\|\mathbf{E}_{h}^{n-1}\|^{2} 
			\\
			\geq&\ \frac{1}{2}\|\mathbf{E}_{h}^{n}\|^{2} - \frac{1}{2}\|\mathbf{E}_{h}^{n-1}\|^{2}.
		\end{aligned}
	\end{equation}
	Similarly, we get
	\begin{equation}\label{e7.3.9}
		\begin{aligned}
			\tau(\bar\partial_{t}\mathbf{H}_{h}^{n},\mathbf{H}_{h}^{n}) 
			\geq&\ \frac{1}{2}\|\mathbf{H}_{h}^{n}\|^{2} - \frac{1}{2}\|\mathbf{H}_{h}^{n-1}\|^{2},
			\\
			\tau(\nabla\cdot\mathbf{u}_{h}^{n},\nabla\cdot\bar\partial_{t}\mathbf{u}_{h}^{n}) 
			\geq&\ \frac{1}{2} \|\nabla\cdot \mathbf{u}_{h}^{n}\|^{2} - \frac{1}{2} \|\nabla\cdot \mathbf{u}_{h}^{n-1}\|^{2},
			\\
			\tau(\nabla\mathbf{u}_{h}^{n},\nabla\bar\partial_{t}\mathbf{u}_{h}^{n}) 
			\geq&\ \frac{1}{2}\|\nabla \mathbf{u}_{h}^{n}\|^{2} - \frac{1}{2}\|\nabla \mathbf{u}_{h}^{n-1}\|^{2},
			\\
			\tau(\bar\partial_{t}p_{h}^{n},p_{h}^{n}) 
			\geq&\ \frac{1}{2}\|p_{h}^{n}\|^{2} - \frac{1}{2}\|p_{h}^{n-1}\|^{2}.
		\end{aligned}
	\end{equation}
	Applying \eqref{e7.3.7}--\eqref{e7.3.9}, the Cauchy--Schwarz and Young inequalities, we obtain
	\begin{align*}
		&\ \epsilon\|\mathbf{E}_{h}^{n}\|^{2} - \epsilon\|\mathbf{E}_{h}^{n-1}\|^{2} 
		+ \mu\|\mathbf{H}_{h}^{n}\|^{2} - \mu\|\mathbf{H}_{h}^{n-1}\|^{2} 
		+ (\lambda+G)\|\nabla\cdot \mathbf{u}_{h}^{n}\|^{2} 
		\\
		&\ - (\lambda+G)\|\nabla\cdot \mathbf{u}_{h}^{n-1}\|^{2} 
		+ G\|\nabla \mathbf{u}_{h}^{n}\|^{2} - G\|\nabla \mathbf{u}_{h}^{n-1}\|^{2}
		+ c_{0}\|p_{h}^{n}\|^{2}  
		\\
		&\ - c_{0}\|p_{h}^{n-1}\|^{2} 
		+ 2\sigma\tau \|\mathbf{E}_{h}^{n}\|^{2} 
		+ 2\kappa\tau \|\nabla p_{h}^{n}\|^{2} 
		\\
		\leq&\ \frac{2L^{2}}{\kappa}\tau\|\mathbf{E}_{h}^{n}\|^{2} 
		+ 2\kappa\tau \|\nabla p_{h}^{n}\|^{2} 
		+ \frac{1}{2\sigma}\tau\|\mathbf{j}(t_{n})\|^{2} 
		+ 2\sigma\tau\|\mathbf{E}_{h}^{n}\|^{2} 
		\\
		&\ + \tau\|g(t_{n})\|^{2} 
		+ \tau\|p_{h}^{n}\|^{2}.
	\end{align*}
	Inductively for $n\geq 1$ in the above inequality, we get
	\begin{equation}\label{e7.3.10}
		\begin{aligned}
			&\ \epsilon\|\mathbf{E}_{h}^{n}\|^{2} + \mu\|\mathbf{H}_{h}^{n}\|^{2} 
			+ (\lambda+G)\|\nabla\cdot \mathbf{u}_{h}^{n}\|^{2} + G\|\nabla \mathbf{u}_{h}^{n}\|^{2} 
			+ c_{0}\|p_{h}^{n}\|^{2}
			\\
			\leq&\ \epsilon\|\mathbf{E}_{h}^{0}\|^{2} + \mu\|\mathbf{H}_{h}^{0}\|^{2} 
			+ (\lambda+G)\|\nabla\cdot \mathbf{u}_{h}^{0}\|^{2} + G\|\nabla \mathbf{u}_{h}^{0}\|^{2} 
			+ c_{0}\|p_{h}^{0}\|^{2}
			\\
			&\ + \frac{1}{2\sigma}\sum_{i=1}^{n}\tau\|\mathbf{j}(t_{i})\|^{2} 
			+ \sum_{i=1}^{n}\tau\|g(t_{i})\|^{2}
			+ \frac{2L^{2}}{\kappa}\sum_{i=1}^{n}\tau\|\mathbf{E}_{h}^{i}\|^{2} 
			+ \sum_{i=1}^{n}\tau\|p_{h}^{i}\|^{2}.
		\end{aligned}
	\end{equation}
	From \eqref{e7.3.5}, it follows that
	\begin{equation}\label{e7.3.11}
		\begin{aligned}
			& \|\mathbf{E}_{h}^{0}\| \leq \|\mathbf{E}_{0}\|,
			\quad
			\|\mathbf{H}_{h}^{0}\| \leq \|\mathbf{H}_{0}\|,
			\quad
			\|\nabla\cdot\mathbf{u}_{h}^{0}\| \leq C\|\mathbf{u}_{0}\|_{1},
			\\
			& \|\nabla\mathbf{u}_{h}^{0}\| \leq C\|\mathbf{u}_{0}\|_{1},
			\quad
			\|p_{h}^{0}\| \leq \|p_{0}\|,
		\end{aligned}
	\end{equation}
	where $C>0$ is a constant independent of $\tau$, $h$ and $\lambda$. 
	Noticing (H2) and applying the discrete Gr\"{o}nwall lemma to \eqref{e7.3.10}, we get
	\begin{equation*}
		\|\mathbf{E}_{h}^{n}\|^{2} + \|\mathbf{H}_{h}^{n}\|^{2} 
		+ \|\nabla\cdot \mathbf{u}_{h}^{n}\|^{2} + \|\nabla \mathbf{u}_{h}^{n}\|^{2} 
		+ \|p_{h}^{n}\|^{2} 
		\leq (C\lambda+C) e^{C\sum_{i=1}^{n}\tau} = C\lambda+C,
	\end{equation*}
	where $C>0$ is a constant independent of $\tau$, $h$ and $\lambda$. 
	The Poincar\'{e} inequality implies 
	\begin{equation*}
		\|\mathbf{u}_{h}^{n}\|_{1}^{2}\leq C\|\nabla \mathbf{u}_{h}^{n}\|^{2}
		\leq C\lambda+C.
	\end{equation*}
	
	Assume that $\mathbf{U}_{i,h}^{n}:= (\mathbf{E}_{i,h}^{n},\mathbf{H}_{i,h}^{n},\mathbf{u}_{i,h}^{n},p_{i,h}^{n})$ $(i=1,2;\ n=1,2,\ldots,N)$ are two solutions of \eqref{e7.3.6} and \eqref{e7.3.5}. Then $\mathbf{U}_{1,h}^{n}-\mathbf{U}_{2,h}^{n}$ satisfies a corresponding homogeneous equations with homogeneous initial and boundary conditions. Thus, we obtain for all $1\leq n\leq N$,
	\begin{equation*}
		\|\mathbf{E}_{1,h}^{n}-\mathbf{E}_{2,h}^{n}\|^{2} 
		+ \|\mathbf{H}_{1,h}^{n}-\mathbf{H}_{2,h}^{n}\|^{2} 
		+ \|\mathbf{u}_{1,h}^{n}-\mathbf{u}_{2,h}^{n}\|_{1}^{2} 
		+ \|p_{1,h}^{n}-p_{2,h}^{n}\|^{2} \leq 0,
	\end{equation*}
	which implies the uniqueness and the proof is completed.
\end{proof}

\begin{remark}\label{r7.3.1}
	This theorem indicates that the numerical solutions are no longer uniformly bounded with respect to $\lambda$.
\end{remark}

\subsection{Error estimates with locking effects}\label{s7.3.2}
In this part, we investigate the error estimates between the exact and numerical solutions to \eqref{e7.2.5} to reveal the causes of the Poisson locking phenomenon in the numerical approximations. 

The errors can be decomposed as, for any $0\leq n\leq N$,
\begin{align*}
	\mathbf{E}(t_{n})-\mathbf{E}_{h}^{n} =&\ (\mathbf{E}(t_{n})-\mathcal{I}_{h}\mathbf{E}(t_{n})) + (\mathcal{I}_{h}\mathbf{E}(t_{n})-\mathbf{E}_{h}^{n}) =: \Xi(t_{n}) + \Pi_{h}^{n}, 
	\\
	\mathbf{H}(t_{n})-\mathbf{H}_{h}^{n} =&\ (\mathbf{H}(t_{n})-\mathcal{P}_{h}\mathbf{H}(t_{n})) + (\mathcal{P}_{h}\mathbf{H}(t_{n})-\mathbf{H}_{h}^{n}) =: \Lambda(t_{n}) + \Theta_{h}^{n}, 
	\\
	\mathbf{u}(t_{n})-\mathbf{u}_{h}^{n} =&\ (\mathbf{u}(t_{n})-\mathcal{Q}_{h}\mathbf{u}(t_{n})) + (\mathcal{Q}_{h}\mathbf{u}(t_{n})-\mathbf{u}_{h}^{n}) =: \boldsymbol{\rho}(t_{n}) + \boldsymbol{\xi}_{h}^{n}, 
	\\
	p(t_{n})-p_{h}^{n} =&\ (p(t_{n})-\mathcal{R}_{h}p(t_{n})) + (\mathcal{R}_{h}p(t_{n})-p_{h}^{n}) =: \gamma(t_{n}) + \eta_{h}^{n}.
\end{align*}
The error estimates for functions $\Xi(t_{n})$, $\Lambda(t_{n})$, $\boldsymbol{\rho}(t_{n})$ and $\gamma(t_{n})$ follow from Lemmas \ref{l7.3.1} and \ref{l7.3.2} and (H3), i.e., there exists a constant $C>0$ independent of $\tau$, $h$ and $\lambda$ such that for all $1\leq n\leq N$,
\begin{equation}\label{e7.3.12}
	\begin{aligned}
		\|\Xi(t_{n})\| 
		\leq&\ Ch^{\min\{l,\beta\}}\|\mathbf{E}(t_{n})\|_{\beta+1} \leq Ch^{\min\{l,\beta\}}, 
		\\
		\|\Lambda(t_{n})\| 
		\leq&\ Ch^{\min\{l,\beta\}}\|\mathbf{H}(t_{n})\|_{\beta} \leq Ch^{\min\{l,\beta\}}, 
		\\
		\|\boldsymbol{\rho}(t_{n})\|_{1} 
		\leq&\ C\lambda h^{\min\{l,\beta\}}\|\mathbf{u}(t_{n})\|_{\beta+1} 
		+ C h^{\min\{l,\beta\}}\|\mathbf{u}(t_{n})\|_{\beta+1} 
		\\
		&\ + Ch^{\min\{l,\beta\}+1}\|p(t_{n})\|_{\beta+1} 
		\leq C\lambda h^{\min\{l,\beta\}}+Ch^{\min\{l,\beta\}}, 
		\\
		\|\gamma(t_{n})\| 
		\leq&\ Ch^{\min\{l,\beta\}+1}\|p(t_{n})\|_{\beta+1} \leq Ch^{\min\{l,\beta\}+1}.
	\end{aligned}
\end{equation}

\begin{theorem}\label{t7.3.2}
	Assume (H1)--(H3). Let $(\mathbf{E},\mathbf{H},\mathbf{u},p)$ and $(\mathbf{E}_{h}^{n},\mathbf{H}_{h}^{n},\mathbf{u}_{h}^{n},p_{h}^{n})$ be the exact and numerical solutions to \eqref{e7.2.5}, respectively. Then there exists a constant $C>0$ independent of $\tau$, $h$ and $\lambda$ such that for all $1\leq n\leq N$,
	\begin{align*}
		&\ \|\mathbf{E}(t_{n})-\mathbf{E}_{h}^{n}\|^{2} 
		+ \|\mathbf{H}(t_{n})-\mathbf{H}_{h}^{n}\|^{2} 
		+ \|\mathbf{u}(t_{n})-\mathbf{u}_{h}^{n}\|_{1}^{2} 
		+ \|p(t_{n})-p_{h}^{n}\|^{2} 
		\\
		\leq&\ C(\tau^{2}+\lambda^{2}h^{2\min\{l,\beta\}}+h^{2\min\{l,\beta\}}).
	\end{align*}
\end{theorem}

\begin{proof}
	Taking $\mathbf{D} = \mathbf{D}_{h}\in \mathbb{E}_{h}$ in the first equation of \eqref{e7.2.5} and noticing the first equation of \eqref{e7.3.6}, we have
	\begin{equation}\label{e7.3.13}
		\begin{aligned}
			&\ (\epsilon\bar\partial_{t}\Pi_{h}^{n},\mathbf{D}_{h}) = (\epsilon\bar\partial_{t}(\mathcal{I}_{h}\mathbf{E}(t_{n})-\mathbf{E}_{h}^{n}),\mathbf{D}_{h}) 
			\\
			=&\ \big(\epsilon\bar\partial_{t}(\mathcal{I}_{h}\mathbf{E}(t_{n})-\mathbf{E}(t_{n})),\mathbf{D}_{h} \big) 
			+ \big( \epsilon(\bar\partial_{t}\mathbf{E}(t_{n})-\frac{\partial}{\partial t}\mathbf{E}(t_{n})),\mathbf{D}_{h}\big) 
			\\
			&\ + \big(\epsilon\frac{\partial}{\partial t}\mathbf{E}(t_{n}),\mathbf{D}_{h}\big) 
			- (\epsilon\bar\partial_{t}\mathbf{E}_{h}^{n},\mathbf{D}_{h})
			\\
			=&\ -(\epsilon\bar\partial_{t}\Xi(t_{n}),\mathbf{D}_{h})
			+ \big(\epsilon(\bar\partial_{t}\mathbf{E}(t_{n})-\frac{\partial}{\partial t}\mathbf{E}(t_{n})),\mathbf{D}_{h}\big) 
			\\
			&\ - (\sigma \Xi(t_{n}),\mathbf{D}_{h}) 
			- (\sigma \Pi_{h}^{n},\mathbf{D}_{h}) 
			+ (\Lambda(t_{n}),\nabla\times \mathbf{D}_{h}) 
			\\
			&\ + (\Theta_{h}^{n},\nabla\times \mathbf{D}_{h}) 
			+ (L\nabla \gamma(t_{n}),\mathbf{D}_{h}) 
			+ (L\nabla \eta_{h}^{n},\mathbf{D}_{h}).
		\end{aligned}
	\end{equation}
	Since $\nabla\times \mathbf{D}_{h}|_{K} \in \mathbb{H}_{h}|_{K}$, we have $(\Lambda(t_{n}),\nabla\times \mathbf{D}_{h}) = 0$. 
	Similarly, by the second equations of \eqref{e7.2.5} and of \eqref{e7.3.6}, respectively, we obtain for any $\mathbf{B}_{h}\in \mathbb{H}_{h}$,
	\begin{equation}\label{e7.3.14}
		\begin{aligned}
			&\ (\mu\bar\partial_{t}\Theta_{h}^{n},\mathbf{B}_{h}) = (\mu\bar\partial_{t}(\mathcal{P}_{h}\mathbf{H}(t_{n})-\mathbf{H}_{h}^{n}),\mathbf{B}_{h}) 
			\\
			=&\ \big(\mu\bar\partial_{t}(\mathcal{P}_{h}\mathbf{H}(t_{n})-\mathbf{H}(t_{n})), \mathbf{B}_{h}\big) 
			+ \big(\mu(\bar\partial_{t}\mathbf{H}(t_{n})-\frac{\partial}{\partial t}\mathbf{H}(t_{n})),\mathbf{B}_{h}\big) 
			\\
			&\ + \big(\mu\frac{\partial}{\partial t}\mathbf{H}(t_{n}),\mathbf{B}_{h}\big) 
			- (\mu\bar\partial_{t}\mathbf{H}_{h}^{n},\mathbf{B}_{h})
			\\
			=&\ - (\mu\bar\partial_{t}\Lambda(t_{n}), \mathbf{B}_{h}) 
			+ \big(\mu(\bar\partial_{t}\mathbf{H}(t_{n})-\frac{\partial}{\partial t}\mathbf{H}(t_{n})),\mathbf{B}_{h}\big) 
			\\
			&\ - (\nabla\times \Xi(t_{n}),\mathbf{B}_{h}) 
			- (\nabla\times \Pi_{h}^{n},\mathbf{B}_{h}).
		\end{aligned}
	\end{equation}
	By \eqref{e7.3.1}, the third equations of \eqref{e7.2.5} and of \eqref{e7.3.6}, respectively, we get for any $\mathbf{v}_{h}\in \mathbb{U}_{h}$,
	\begin{equation}\label{e7.3.15}
		\begin{aligned}
			&\ a_{\lambda}(\boldsymbol{\xi}_{h}^{n},\mathbf{v}_{h}) - (\eta_{h}^{n},\alpha\nabla\cdot\mathbf{v}_{h}) 
			\\
			=&\ a_{\lambda}(\mathcal{Q}_{h}\mathbf{u}(t_{n}),\mathbf{v}_{h}) 
			- (\mathcal{R}_{h}p(t_{n}),\alpha\nabla\cdot\mathbf{v}_{h}) 
			- a_{\lambda}(\mathbf{u}_{h}^{n},\mathbf{v}_{h}) 
			+ (p_{h}^{n},\alpha\nabla\cdot\mathbf{v}_{h})
			\\
			=&\ a_{\lambda}(\mathbf{u}(t_{n}),\mathbf{v}_{h}) 
			- (p(t_{n}),\alpha\nabla\cdot\mathbf{v}_{h}) 
			- a_{\lambda}(\mathbf{u}_{h}^{n},\mathbf{v}_{h}) 
			+ (p_{h}^{n},\alpha\nabla\cdot\mathbf{v}_{h})
			=0.
		\end{aligned}
	\end{equation}
	By virtue of \eqref{e7.3.2}, the last equations of \eqref{e7.2.5} and of \eqref{e7.3.6}, respectively, we have for any $q_{h}\in \mathbb{P}_{h}$,
	\begin{equation}\label{e7.3.16}
		\begin{aligned}
			&\ (c_{0}\bar\partial_{t}\eta_{h}^{n},q_{h}) + (\alpha\nabla\cdot\bar\partial_{t}\boldsymbol{\xi}_{h}^{n},q_{h}) 
			+ (\kappa\nabla\eta_{h}^{n},\nabla q_{h})
			\\
			=&\ (c_{0}\bar\partial_{t}\mathcal{R}_{h}p(t_{n}),q_{h}) 
			+ (\alpha\nabla\cdot\bar\partial_{t}\mathcal{Q}_{h}\mathbf{u}(t_{n}),q_{h}) 
			+ (\kappa\nabla\mathcal{R}_{h}p(t_{n}),\nabla q_{h}) 
			\\
			&\ - (c_{0}\bar\partial_{t}p_{h}^{n},q_{h}) 
			- (\alpha\nabla\cdot\bar\partial_{t}\mathbf{u}_{h}^{n},q_{h}) 
			- (\kappa\nabla p_{h}^{n},\nabla q_{h})
			\\
			=&\ \big(c_{0}\bar\partial_{t}(\mathcal{R}_{h}p(t_{n})-p(t_{n})),q_{h}\big) 
			+ \big(c_{0}(\bar\partial_{t}p(t_{n})-\frac{\partial}{\partial t}p(t_{n})),q_{h}\big) 
			\\
			&\ + \big(\alpha\nabla\cdot\bar\partial_{t}(\mathcal{Q}_{h}\mathbf{u}(t_{n})-\mathbf{u}(t_{n})),q_{h}\big) 
			+ \big(\alpha\nabla\cdot(\bar\partial_{t}\mathbf{u}(t_{n})-\frac{\partial}{\partial t}\mathbf{u}(t_{n})),q_{h}\big) 
			\\
			&\ + \big(c_{0}\frac{\partial}{\partial t}p(t_{n}),q_{h}\big) 
			+ \big(\alpha\nabla\cdot\frac{\partial}{\partial t}\mathbf{u}(t_{n}),q_{h}\big) 
			+ (\kappa\nabla p(t_{n}),\nabla q_{h}) 
			\\
			&\ - (c_{0}\bar\partial_{t}p_{h}^{n},q_{h}) 
			- (\alpha\nabla\cdot\bar\partial_{t}\mathbf{u}_{h}^{n},q_{h}) 
			- (\kappa\nabla p_{h}^{n},\nabla q_{h})
			\\
			=&\ - (c_{0}\bar\partial_{t}\gamma(t_{n}),q_{h}) 
			+ \big(c_{0}(\bar\partial_{t}p(t_{n}) - \frac{\partial}{\partial t}p(t_{n})),q_{h}\big) 
			- (\alpha\nabla\cdot\bar\partial_{t}\boldsymbol{\rho}(t_{n}),q_{h}) 
			\\
			&\ + \big(\alpha\nabla\cdot(\bar\partial_{t}\mathbf{u}(t_{n})-\frac{\partial}{\partial t}\mathbf{u}(t_{n})),q_{h}\big) 
			+ (L\Xi(t_{n}),\nabla q_{h}) 
			+ (L\Pi_{h}^{n},\nabla q_{h}).
		\end{aligned}
	\end{equation}
	To simplify notations, define
	\begin{align*}
		& \bm{e}_{1}(t_{n})= \bar\partial_{t}\mathbf{E}(t_{n})-\frac{\partial}{\partial t}\mathbf{E}(t_{n}), 
		\quad
		\bm{e}_{2}(t_{n})= \bar\partial_{t}\mathbf{H}(t_{n})-\frac{\partial}{\partial t}\mathbf{H}(t_{n}),
		\\
		& e_{3}(t_{n})= \bar\partial_{t}p(t_{n}) - \frac{\partial}{\partial t}p(t_{n}), 
		\quad
		\bm{e}_{4}(t_{n})= \bar\partial_{t}\mathbf{u}(t_{n})-\frac{\partial}{\partial t}\mathbf{u}(t_{n}).
	\end{align*}
	Taking $\mathbf{D}_{h}=\Pi_{h}^{n}$, $\mathbf{B}_{h}=\Theta_{h}^{n}$, $\mathbf{v}_{h}=\bar\partial_{t}\boldsymbol{\xi}_{h}^{n}$ and $q_{h}=\eta_{h}^{n}$ in \eqref{e7.3.13}--\eqref{e7.3.16}, respectively, and summing these up, we obtain an error equation
	\begin{equation}\label{e7.3.17}
		\begin{aligned}
			&\ (\epsilon\bar\partial_{t}\Pi_{h}^{n},\Pi_{h}^{n}) 
			+ (\mu\bar\partial_{t}\Theta_{h}^{n},\Theta_{h}^{n}) 
			+ a_{\lambda}(\boldsymbol{\xi}_{h}^{n},\bar\partial_{t}\boldsymbol{\xi}_{h}^{n}) 
			\\
			&\ + (c_{0}\bar\partial_{t}\eta_{h}^{n}, \eta_{h}^{n}) 
			+ (\kappa\nabla \eta_{h}^{n},\nabla \eta_{h}^{n}) 
			\\
			=&\ -(\epsilon\bar\partial_{t}\Xi(t_{n}),\Pi_{h}^{n})
			+ (\epsilon\bm{e}_{1}(t_{n}),\Pi_{h}^{n}) 
			- (\sigma \Xi(t_{n}),\Pi_{h}^{n}) 
			- (\sigma \Pi_{h}^{n},\Pi_{h}^{n}) 
			\\
			&\ + (L\nabla \gamma(t_{n}),\Pi_{h}^{n}) 
			+ 2(L\nabla \eta_{h}^{n},\Pi_{h}^{n}) 
			- (\mu\bar\partial_{t}\Lambda(t_{n}), \Theta_{h}^{n}) 
			+ (\mu\bm{e}_{2}(t_{n}),\Theta_{h}^{n}) 
			\\
			&\ - (\nabla\times \Xi(t_{n}),\Theta_{h}^{n}) 
			- (c_{0}\bar\partial_{t}\gamma(t_{n}),\eta_{h}^{n}) 
			+ (c_{0}e_{3}(t_{n}),\eta_{h}^{n}) 
			\\
			&\ - (\alpha\nabla\cdot\bar\partial_{t}\boldsymbol{\rho}(t_{n}),\eta_{h}^{n}) 
			+ (\alpha\nabla\cdot\bm{e}_{4}(t_{n}),\eta_{h}^{n}) 
			+ (L\Xi(t_{n}),\nabla \eta_{h}^{n}).
		\end{aligned}
	\end{equation}
	In a similar way as proving \eqref{e7.3.8}, we have
	\begin{equation}\label{e7.3.18}
		\begin{aligned}
			\tau(\bar\partial_{t}\Pi_{h}^{n},\Pi_{h}^{n}) 
			\geq&\ \frac{1}{2}\|\Pi_{h}^{n}\|^{2} - \frac{1}{2}\|\Pi_{h}^{n-1}\|^{2},
			\\
			\tau(\bar\partial_{t}\Theta_{h}^{n},\Theta_{h}^{n}) 
			\geq&\ \frac{1}{2}\|\Theta_{h}^{n}\|^{2} - \frac{1}{2}\|\Theta_{h}^{n-1}\|^{2},
			\\
			\tau a_{\lambda}(\boldsymbol{\xi}_{h}^{n},\bar\partial_{t}\boldsymbol{\xi}_{h}^{n})
			\geq&\ \frac{\lambda+G}{2}\big(\|\nabla\cdot\boldsymbol{\xi}_{h}^{n}\|^{2}-\|\nabla\cdot\boldsymbol{\xi}_{h}^{n-1}\|^{2}\big) 
			\\
			&\ + \frac{G}{2}\big(\|\nabla\boldsymbol{\xi}_{h}^{n}\|^{2}-\|\nabla\boldsymbol{\xi}_{h}^{n-1}\|^{2}\big),
			\\
			\tau(\bar\partial_{t}\eta_{h}^{n}, \eta_{h}^{n}) 
			\geq&\ \frac{1}{2}\|\eta_{h}^{n}\|^{2} - \frac{1}{2}\|\eta_{h}^{n-1}\|^{2}.
		\end{aligned}
	\end{equation}
	According to \eqref{e7.3.17}, \eqref{e7.3.18}, the Cauchy--Schwarz and Young inequalities, we get
	\begin{align*}
		&\ \frac{\epsilon}{2}\big(\|\Pi_{h}^{n}\|^{2}-\|\Pi_{h}^{n-1}\|^{2}\big) 
		+ \frac{\mu}{2}\big(\|\Theta_{h}^{n}\|^{2}-\|\Theta_{h}^{n-1}\|^{2}\big) 
		\\
		&\ + \frac{\lambda+G}{2}\big(\|\nabla\cdot \boldsymbol{\xi}_{h}^{n}\|^{2}-\|\nabla\cdot \boldsymbol{\xi}_{h}^{n-1}\|^{2}\big)
		+ \frac{G}{2}\big(\|\nabla \boldsymbol{\xi}_{h}^{n}\|^{2}-\|\nabla \boldsymbol{\xi}_{h}^{n-1}\|^{2}\big) 
		\\
		&\ + \frac{c_{0}}{2}\big(\|\eta_{h}^{n}\|^{2}-\|\eta_{h}^{n-1}\|^{2}\big) 
		+ \kappa\tau\|\nabla \eta_{h}^{n}\|^{2}
		\\
		\leq&\ \frac{\epsilon}{2}\tau\|\bar\partial_{t}\Xi(t_{n})\|^{2} 
		+ \frac{\epsilon}{2}\tau\|\Pi_{h}^{n}\|^{2} 
		+ \frac{\epsilon}{2}\tau\|\bm{e}_{1}(t_{n})\|^{2} 
		+ \frac{\epsilon}{2}\tau\|\Pi_{h}^{n}\|^{2} 
		+ \frac{\sigma}{2}\tau\|\Xi(t_{n})\|^{2} 
		\\
		&\ + \frac{\sigma}{2}\tau\|\Pi_{h}^{n}\|^{2} 
		+ \sigma\tau\|\Pi_{h}^{n}\|^{2} 
		+ \frac{L}{2}\tau\|\nabla\gamma(t_{n})\|^{2} 
		+ \frac{L}{2}\tau\|\Pi_{h}^{n}\|^{2} 
		+ \frac{\kappa}{2}\tau\|\nabla\eta_{h}^{n}\|^{2} 
		\\
		&\ + \frac{2L^{2}}{\kappa}\tau\|\Pi_{h}^{n}\|^{2} 
		+ \frac{\mu}{2}\tau\|\bar\partial_{t}\Lambda(t_{n})\|^{2} 
		+ \frac{\mu}{2}\tau\|\Theta_{h}^{n}\|^{2} 
		+ \frac{\mu}{2}\tau\|\bm{e}_{2}(t_{n})\|^{2} 
		+ \frac{\mu}{2}\tau\|\Theta_{h}^{n}\|^{2} 
		\\
		&\ + \frac{1}{2}\tau\|\nabla\times\Xi(t_{n})\|^{2} 
		+ \frac{1}{2}\tau\|\Theta_{h}^{n}\|^{2} 
		+ \frac{c_{0}}{2}\tau\|\bar\partial_{t}\gamma(t_{n})\|^{2} 
		+ \frac{c_{0}}{2}\tau\|\eta_{h}^{n}\|^{2} 
		\\
		&\ + \frac{c_{0}}{2}\tau\|e_{3}(t_{n})\|^{2} 
		+ \frac{c_{0}}{2}\tau\|\eta_{h}^{n}\|^{2} 
		+ \frac{\alpha}{2}\tau\|\bar\partial_{t}\boldsymbol{\rho}(t_{n})\|_{1}^{2} 
		+ \frac{\alpha}{2}\tau\|\eta_{h}^{n}\|^{2} 
		\\
		&\ + \frac{\alpha}{2}\tau\|\bm{e}_{4}(t_{n})\|_{1}^{2} 
		+ \frac{\alpha}{2}\tau\|\eta_{h}^{n}\|^{2} 
		+ \frac{L^{2}}{2\kappa}\tau\|\Xi(t_{n})\|^{2} 
		+ \frac{\kappa}{2}\tau\|\nabla \eta_{h}^{n}\|^{2}.
	\end{align*}
	Inductively for $n\geq 1$ in the above inequality, we have
	\begin{equation}\label{e7.3.19}
		\begin{aligned}
			&\ \epsilon\|\Pi_{h}^{n}\|^{2} 
			+ \mu\|\Theta_{h}^{n}\|^{2} 
			+ (\lambda+G)\|\nabla\cdot\boldsymbol{\xi}_{h}^{n}\|^{2} 
			+ G\|\nabla\boldsymbol{\xi}_{h}^{n}\|^{2} 
			+ c_{0}\|\eta_{h}^{n}\|^{2} 
			\\
			\leq&\ \epsilon\|\Pi_{h}^{0}\|^{2} 
			+ \mu\|\Theta_{h}^{0}\|^{2} 
			+ (\lambda+G)\|\nabla\cdot\boldsymbol{\xi}_{h}^{0}\|^{2} 
			+ G\|\nabla\boldsymbol{\xi}_{h}^{0}\|^{2} 
			+ c_{0}\|\eta_{h}^{0}\|^{2} 
			\\
			&\ + \epsilon\sum_{i=1}^{n}\tau\|\bar\partial_{t}\Xi(t_{i})\|^{2} 
			+ \epsilon\sum_{i=1}^{n}\tau\|\bm{e}_{1}(t_{i})\|^{2}
			+ (\sigma+\frac{L^{2}}{\kappa})\sum_{i=1}^{n}\tau\|\Xi(t_{i})\|^{2} 
			\\
			&\ + L\sum_{i=1}^{n}\tau\|\nabla \gamma(t_{i})\|^{2} 
			+ \mu\sum_{i=1}^{n}\tau\|\bar\partial_{t}\Lambda(t_{i})\|^{2} 
			+ \mu\sum_{i=1}^{n}\tau\|\bm{e}_{2}(t_{i})\|^{2} 
			\\
			&\ + \sum_{i=1}^{n}\tau\|\nabla\times \Xi(t_{i})\|^{2}
			+ c_{0}\sum_{i=1}^{n}\tau\|\bar\partial_{t}\gamma(t_{i})\|^{2} 
			+ c_{0}\sum_{i=1}^{n}\tau\|e_{3}(t_{i})\|^{2} 
			\\
			&\ + \alpha\sum_{i=1}^{n}\tau\|\bar\partial_{t}\boldsymbol{\rho}(t_{i})\|_{1}^{2} 
			+ \alpha\sum_{i=1}^{n}\tau\|\bm{e}_{4}(t_{i})\|_{1}^{2} 
			\\
			&\ + (2\epsilon+3\sigma+L+\frac{4L^{2}}{\kappa})\sum_{i=1}^{n}\tau\|\Pi_{h}^{i}\|^{2} 
			+ (2\mu+1)\sum_{i=1}^{n}\tau\|\Theta_{h}^{i}\|^{2} 
			\\
			&\ + (2c_{0}+2\alpha)\sum_{i=1}^{n}\tau\|\eta_{h}^{i}\|^{2}.
		\end{aligned}
	\end{equation}
	Noticing \eqref{e7.3.5}, and then utilizing Lemmas \ref{l7.3.1} and \ref{l7.3.2} yields
	\begin{equation}\label{e7.3.20}
		\begin{aligned}
			\|\Pi_{h}^{0}\|  
			\leq&\ \|\mathcal{I}_{h}\mathbf{E}_{0}\!-\!\mathbf{E}_{0}\| 
			+ \|\mathbf{E}_{0}\!-\!\mathfrak{P}_{h}\mathbf{E}_{0}\| 
			\leq Ch^{\min\{l,\beta\}}\|\mathbf{E}_{0}\|_{\beta+1},
			\\
			\|\Theta_{h}^{0}\| =&\ 0, 
			\\
			(\lambda+G)\|\nabla\cdot \boldsymbol{\xi}_{h}^{0}\|^{2} 
			=&\ (\lambda+G)\|\nabla\cdot (\mathcal{Q}_{h}\mathbf{u}_{0}-\mathscr{P}_{h}\mathbf{u}_{0})\|^{2}
			\\
			\leq&\ C\lambda^{2} h^{2\min\{l,\beta\}}\|\mathbf{u}_{0}\|_{\beta+1}^{2} 
			+ Ch^{2\min\{l,\beta\}}\|\mathbf{u}_{0}\|_{\beta+1}^{2} 
			\\
			&\ + Ch^{2(\min\{l,\beta\}+1)}\|p_{0}\|_{\beta+1}^{2},
			\\
			\|\nabla \boldsymbol{\xi}_{h}^{0}\| 
			\leq&\ \|\mathcal{Q}_{h}\mathbf{u}_{0}-\mathbf{u}_{0}\|_{1} 
			+ \|\mathbf{u}_{0}-\mathscr{P}_{h}\mathbf{u}_{0}\|_{1}
			\\
			\leq&\ C\lambda h^{\min\{l,\beta\}}\|\mathbf{u}_{0}\|_{\beta+1} 
			+ Ch^{\min\{l,\beta\}}\|\mathbf{u}_{0}\|_{\beta+1} 
			\\
			&\ + Ch^{\min\{l,\beta\}+1}\|p_{0}\|_{\beta+1},
			\\
			\|\eta_{h}^{0}\| \leq&\ \|\mathcal{R}_{h}p_{0}\!-\!p_{0}\|+\|p_{0}\!-\!P_{h}p_{0}\| 
			\leq Ch^{\min\{l,\beta\}+1}\|p_{0}\|_{\beta+1}.
		\end{aligned}
	\end{equation}
	where $C>0$ is a constant independent of $\tau$, $h$ and $\lambda$. 
	By the Newton-Leibniz formula, we get
	\begin{equation*}
		\bar\partial_{t}\Xi(t_{i}) = \bar\partial_{t}\mathbf{E}(t_{i})-\bar\partial_{t}\mathcal{I}_{h}\mathbf{E}(t_{i}) 
		= \frac{1}{\tau}\int_{t_{i-1}}^{t_{i}}{\big(\frac{\partial}{\partial t} \mathbf{E}(\theta)-\mathcal{I}_{h}\frac{\partial}{\partial t}\mathbf{E}(\theta)\big)}{\, \mathrm{d}\theta}.
	\end{equation*}
	Utilizing the Cauchy--Schwarz inequality and Lemma \ref{l7.3.2}, we have
	\begin{equation}\label{e7.3.21}
		\begin{aligned}
			&\ \|\bar\partial_{t}\Xi(t_{i})\|^{2} 
			\leq \frac{1}{\tau^{2}}\Big(\int_{t_{i-1}}^{t_{i}}{\|\frac{\partial}{\partial t} \mathbf{E}(\theta)-\mathcal{I}_{h}\frac{\partial}{\partial t}\mathbf{E}(\theta)\|}{\, \mathrm{d}\theta}\Big)^{2}
			\\
			\leq&\ \frac{1}{\tau}\int_{t_{i-1}}^{t_{i}}{\|\frac{\partial}{\partial t} \mathbf{E}(\theta)-\mathcal{I}_{h}\frac{\partial}{\partial t} \mathbf{E}(\theta)\|^{2}}{\, \mathrm{d}\theta}
			\\
			\leq&\ C\frac{1}{\tau} h^{2\min\{l,\beta\}} \int_{t_{i-1}}^{t_{i}}{\|\frac{\partial}{\partial t} \mathbf{E}(\theta)\|_{\beta+1}^{2}}{\, \mathrm{d}\theta}.
		\end{aligned}
	\end{equation}
	Similarly, applying Lemmas \ref{l7.3.1} and \ref{l7.3.2}, we obtain
	\begin{equation}\label{e7.3.22}
		\begin{aligned}
			\|\bar\partial_{t}\Lambda(t_{i})\|^{2} \leq&\ C\frac{1}{\tau} h^{2\min\{l,\beta\}} \int_{t_{i-1}}^{t_{i}}{\|\frac{\partial}{\partial t} \mathbf{H}(\theta)\|_{\beta}^{2}}{\, \mathrm{d}\theta},
			\\
			\|\bar\partial_{t}\gamma(t_{i})\|^{2} \leq&\ C\frac{1}{\tau} h^{2(\min\{l,\beta\}+1)} \int_{t_{i-1}}^{t_{i}}{\|\frac{\partial}{\partial t} p(\theta)\|_{\beta+1}^{2}}{\, \mathrm{d}\theta}, 
			\\
			\|\bar\partial_{t}\boldsymbol{\rho}(t_{i})\|_{1}^{2} \leq&\ C\frac{1}{\tau} \lambda^{2}h^{2\min\{l,\beta\}} \int_{t_{i-1}}^{t_{i}}{\|\frac{\partial}{\partial t} \mathbf{u}(\theta)\|_{\beta+1}^{2}}{\, \mathrm{d}\theta} 
			\\
			&\ + C\frac{1}{\tau} h^{2\min\{l,\beta\}} \int_{t_{i-1}}^{t_{i}}{\|\frac{\partial}{\partial t} \mathbf{u}(\theta)\|_{\beta+1}^{2}}{\, \mathrm{d}\theta} 
			\\
			&\ + C\frac{1}{\tau} h^{2(\min\{l,\beta\}+1)} \int_{t_{i-1}}^{t_{i}}{\|\frac{\partial}{\partial t} p(\theta)\|_{\beta+1}^{2}}{\, \mathrm{d}\theta}.
		\end{aligned}
	\end{equation}
	By virtue of the Taylor expansion, we have
	\begin{equation*}
		\bm{e}_{1}(t_{i}) = \bar\partial_{t}\mathbf{E}(t_{i})-\frac{\partial}{\partial t}\mathbf{E}(t_{i}) 
		= \frac{1}{\tau}\int_{t_{i-1}}^{t_{i}}{(t_{i-1}-\theta)\frac{\partial^{2}}{\partial t^{2}} \mathbf{E}(\theta)}{\, \mathrm{d}\theta}.
	\end{equation*}
	By the Cauchy--Schwarz inequality, we get
	\begin{equation}\label{e7.3.23}
		\|\bm{e}_{1}(t_{i})\|^{2}
		\!\leq\! \frac{1}{\tau^{2}} \Big(\int_{t_{i-1}}^{t_{j}}{(t_{i-1}-\theta)\|\frac{\partial^{2}}{\partial t^{2}} \mathbf{E}(\theta)\|}{\, \mathrm{d}\theta}\Big)^{2}
		\\
		\!\leq\! \tau\int_{t_{i-1}}^{t_{i}}{\|\frac{\partial^{2}}{\partial t^{2}} \mathbf{E}(\theta)\|^{2}}{\, \mathrm{d}\theta}.
	\end{equation}
	In a similar way, we obtain
	\begin{equation}\label{e7.3.24}
		\begin{aligned}
			& \|\bm{e}_{2}(t_{i})\|^{2}
			\leq \tau\int_{t_{i-1}}^{t_{i}}{\|\frac{\partial^{2}}{\partial t^{2}} \mathbf{H}(\theta)\|^{2}}{\, \mathrm{d}\theta},
			\quad
			\|e_{3}(t_{i})\|^{2}
			\leq \tau\int_{t_{i-1}}^{t_{i}}{\|\frac{\partial^{2}}{\partial t^{2}} p(\theta)\|^{2}}{\, \mathrm{d}\theta},
			\\
			& \|\bm{e}_{4}(t_{i})\|_{1}^{2}
			\leq \tau\int_{t_{i-1}}^{t_{i}}{\|\frac{\partial^{2}}{\partial t^{2}} \mathbf{u}(\theta)\|_{1}^{2}}{\, \mathrm{d}\theta}.
		\end{aligned}
	\end{equation}
	According to \eqref{e7.3.19}--\eqref{e7.3.24}, Lemma \ref{l7.3.2}, (H2) and (H3), we get
	\begin{align*}
		&\ \epsilon\|\Pi_{h}^{n}\|^{2} 
		+ \mu\|\Theta_{h}^{n}\|^{2} 
		+ (\lambda+G)\|\nabla\cdot\boldsymbol{\xi}_{h}^{n}\|^{2} 
		+ G\|\nabla\boldsymbol{\xi}_{h}^{n}\|^{2} 
		+ c_{0}\|\eta_{h}^{n}\|^{2} 
		\\
		\leq&\ \epsilon Ch^{2\min\{l,\beta\}}\|\mathbf{E}_{0}\|_{\beta+1}^{2} 
		+ C\lambda^{2} h^{2\min\{l,\beta\}}\|\mathbf{u}_{0}\|_{\beta+1}^{2} 
		+ Ch^{2\min\{l,\beta\}}\|\mathbf{u}_{0}\|_{\beta+1}^{2} 
		\\
		&\ + Ch^{2(\min\{l,\beta\}+1)}\|p_{0}\|_{\beta+1}^{2} 
		+ GC\lambda^{2} h^{2\min\{l,\beta\}}\|\mathbf{u}_{0}\|_{\beta+1}^{2} 
		\\
		&\ + GCh^{2\min\{l,\beta\}}\|\mathbf{u}_{0}\|_{\beta+1}^{2} 
		+ GCh^{2(\min\{l,\beta\}+1)}\|p_{0}\|_{\beta+1}^{2} 
		\\
		&\ + c_{0}Ch^{2(\min\{l,\beta\}+1)}\|p_{0}\|_{\beta+1}^{2} 
		+\epsilon Ch^{2\min\{l,\beta\}} \int_{0}^{t_{n}}{\|\frac{\partial}{\partial t} \mathbf{E}(\theta)\|_{\beta+1}^{2}}{\, \mathrm{d}\theta} 
		\\
		&\ + \epsilon \tau^{2}\int_{0}^{t_{n}}{\|\frac{\partial^{2}}{\partial t^{2}} \mathbf{E}(\theta)\|^{2}}{\, \mathrm{d}\theta} 
		+ (\sigma+\frac{L^{2}}{\kappa})Ch^{2\min\{l,\beta\}}\sum_{i=1}^{n}\tau\|\mathbf{E}(t_{i})\|_{\beta+1}^{2} 
		\\
		&\ + LCh^{2\min\{l,\beta\}}\sum_{i=1}^{n}\tau\|p(t_{i})\|_{\beta+1}^{2} 
		+ \mu Ch^{2\min\{l,\beta\}} \int_{0}^{t_{n}}{\|\frac{\partial}{\partial t} \mathbf{H}(\theta)\|_{\beta}^{2}}{\, \mathrm{d}\theta} 
		\\
		&\ + \mu \tau^{2}\int_{0}^{t_{n}}{\|\frac{\partial^{2}}{\partial t^{2}} \mathbf{H}(\theta)\|^{2}}{\, \mathrm{d}\theta} 
		+ Ch^{2\min\{l,\beta\}}\sum_{i=1}^{n}\tau\|\mathbf{E}(t_{i})\|_{\beta+1}^{2} 
		\\
		&\ + c_{0}Ch^{2(\min\{l,\beta\}+1)} \int_{0}^{t_{n}}{\|\frac{\partial}{\partial t} p(\theta)\|_{\beta+1}^{2}}{\, \mathrm{d}\theta} 
		+ c_{0}\tau^{2}\int_{0}^{t_{n}}{\|\frac{\partial^{2}}{\partial t^{2}} p(\theta)\|^{2}}{\, \mathrm{d}\theta} 
		\\
		&\ + \alpha C\lambda^{2}h^{2\min\{l,\beta\}} \int_{0}^{t_{n}}{\|\frac{\partial}{\partial t} \mathbf{u}(\theta)\|_{\beta+1}^{2}}{\, \mathrm{d}\theta} 
		\\
		&\ + \alpha Ch^{2\min\{l,\beta\}} \int_{0}^{t_{n}}{\|\frac{\partial}{\partial t} \mathbf{u}(\theta)\|_{\beta+1}^{2}}{\, \mathrm{d}\theta} 
		\\
		&\ + \alpha Ch^{2(\min\{l,\beta\}+1)} \int_{0}^{t_{n}}{\|\frac{\partial}{\partial t} p(\theta)\|_{\beta+1}^{2}}{\, \mathrm{d}\theta} 
		+ \alpha \tau^{2}\int_{0}^{t_{n}}{\|\frac{\partial^{2}}{\partial t^{2}} \mathbf{u}(\theta)\|_{1}^{2}}{\, \mathrm{d}\theta}
		\\
		&\ + (2\epsilon+3\sigma+L+\frac{4L^{2}}{\kappa})\sum_{i=1}^{n}\tau\|\Pi_{h}^{i}\|^{2} 
		+ (2\mu+1)\sum_{i=1}^{n}\tau\|\Theta_{h}^{i}\|^{2} 
		\\
		&\ + (2c_{0}+2\alpha)\sum_{i=1}^{n}\tau\|\eta_{h}^{i}\|^{2} 
		\\
		\leq&\ C(\tau^{2}+\lambda^{2}h^{2\min\{l,\beta\}}+h^{2\min\{l,\beta\}}) 
		+ C\sum_{i=1}^{n}\tau(\|\Pi_{h}^{i}\|^{2}+\|\Theta_{h}^{i}\|^{2}+\|\eta_{h}^{i}\|^{2}),
	\end{align*}
	where $C>0$ is a constant independent of $\tau$, $h$ and $\lambda$. 
	Applying the discrete Gr\"{o}nwall lemma to the above inequality, we obtain
	\begin{align*}
		&\ \|\Pi_{h}^{n}\|^{2} 
		+ \|\Theta_{h}^{n}\|^{2} 
		+ \|\nabla\cdot\boldsymbol{\xi}_{h}^{n}\|^{2} 
		+ \|\nabla\boldsymbol{\xi}_{h}^{n}\|^{2} 
		+ \|\eta_{h}^{n}\|^{2} 
		\\
		\leq&\ C (\tau^{2}+\lambda^{2}h^{2\min\{l,\beta\}}+h^{2\min\{l,\beta\}}) e^{C\sum_{i=1}^{n}\tau} 
		\\
		\leq&\ C (\tau^{2}+\lambda^{2}h^{2\min\{l,\beta\}}+h^{2\min\{l,\beta\}}),
	\end{align*}
	which together with the Poincar\'{e} inequality leads to
	\begin{equation*}
		\|\Pi_{h}^{n}\|^{2} 
		+ \|\Theta_{h}^{n}\|^{2} 
		+ \|\boldsymbol{\xi}_{h}^{n}\|_{1}^{2}
		+ \|\eta_{h}^{n}\|^{2} 
		\leq C (\tau^{2}+\lambda^{2}h^{2\min\{l,\beta\}}+h^{2\min\{l,\beta\}}).
	\end{equation*}
	Then, combining with \eqref{e7.3.12}, we complete the proof.
\end{proof}

\begin{remark}\label{r7.3.2}
	This theorem indicates that the errors become large as $\lambda\to\infty$, hence the spatial convergence order deteriorates. This is the reason that locking effect occurs to conforming finite element approximations.
\end{remark}

\section{High-order locking-free FEM}\label{s7.4}
In this section, we are devoted to constructing a high-order locking-free finite element approximation and deriving uniform error estimates with respect to $\lambda$. We first introduce a five-field model and prove uniform boundedness of solutions with respect to $\lambda$. Then, we employ the backward Euler method and FEM with N\'{e}d\'{e}lec \cite{Ne1980}, Lagrange and Taylor--Hood \cite{GiRa1986} elements to construct a fully discretized scheme. Finally, we derive uniform error estimates with respect to $\lambda$, indicating the locking-free property.

\subsection{Five-field model and uniform a priori estimates}\label{s7.4.1}
In this part, we convert \eqref{e7.2.1} to a five-field model such that we can propose a locking-free numerical method, and then study uniform boundedness of solutions with respect to $\lambda$. 
Define a new unknown $\varphi = \alpha p - \lambda\nabla\cdot\mathbf{u}$. 
Then, \eqref{e7.2.1} can be rewritten as a five-field model
\begin{equation}\label{e7.4.1}
	\begin{aligned}
		& \epsilon\frac{\partial}{\partial t}\mathbf{E} + \sigma \mathbf{E} - \nabla\times \mathbf{H} - L\nabla p = \mathbf{j},
		\\
		& \mu\frac{\partial}{\partial t}\mathbf{H} + \nabla\times \mathbf{E} = \mathbf{0},
		\\
		& -G\nabla(\nabla\cdot \mathbf{u}) - G\Delta \mathbf{u} + \nabla \varphi = \mathbf{0},
		\\
		& \frac{1}{\lambda}\varphi - \frac{\alpha}{\lambda} p + \nabla\cdot\mathbf{u} = 0,
		\\
		& (c_{0}+\frac{\alpha^{2}}{\lambda})\frac{\partial}{\partial t}p 
		- \frac{\alpha}{\lambda}\frac{\partial}{\partial t}\varphi 
		- \kappa\Delta p + L\nabla\cdot \mathbf{E} = g.
	\end{aligned}
\end{equation}
Moreover, the initial value of $\varphi$ is
\begin{equation}\label{e7.4.2}
	\varphi_{0}  
	= \alpha p_{0} - \lambda\nabla\cdot\mathbf{u}_{0}.
\end{equation}

The variational equation of five-field model \eqref{e7.4.1}, \eqref{e7.2.2}, \eqref{e7.4.2} and \eqref{e7.2.3} is to find $(\mathbf{E}(t),\mathbf{H}(t),\mathbf{u}(t),\varphi(t),p(t))\in \bm{H}_{0}(\mathbf{curl}) \times \bm{L}^{2}(\mathscr{D}) \times \bm{H}^{1}_{0} \times L^{2}(\mathscr{D}) \times H^{1}_{0}$ satisfying the initial conditions \eqref{e7.2.2} and \eqref{e7.4.2} such that
\begin{equation}\label{e7.4.3}
	\begin{aligned}
		& \big(\epsilon\frac{\partial}{\partial t}\mathbf{E}(t), \mathbf{D}\big) 
		+ \big(\sigma \mathbf{E}(t), \mathbf{D}\big) 
		- \big(\mathbf{H}(t), \nabla\!\times\! \mathbf{D}\big) 
		- \big(L\nabla p(t), \mathbf{D}\big) 
		= \big(\mathbf{j}(t), \mathbf{D}\big),
		\\
		& \big(\mu\frac{\partial}{\partial t}\mathbf{H}(t), \mathbf{B}\big) 
		+ \big(\nabla\times \mathbf{E}(t), \mathbf{B}\big) = 0,
		\\
		& \big(G\nabla\cdot \mathbf{u}(t), \nabla\cdot \mathbf{v}\big) 
		+ \big(G\nabla \mathbf{u}(t), \nabla \mathbf{v}\big) 
		- \big(\varphi(t), \nabla\cdot \mathbf{v}\big) = 0,
		\\
		& \big(\frac{1}{\lambda}\varphi(t), \psi\big) 
		- \big(\frac{\alpha}{\lambda} p(t), \psi\big) 
		+ \big(\nabla\cdot \mathbf{u}(t), \psi\big) 
		= 0,
		\\
		& \big((c_{0}+\frac{\alpha^{2}}{\lambda})\frac{\partial}{\partial t}p(t), q\big) 
		- \big(\frac{\alpha}{\lambda}\frac{\partial}{\partial t}\varphi(t), q\big) 
		+ \big(\kappa\nabla p(t), \nabla q\big) 
		- \big(L\mathbf{E}(t), \nabla q\big) 
		\\
		&\ \qquad\qquad = \big(g(t), q\big),
	\end{aligned}
\end{equation}
for every $t\in (0,T]$ and all $(\mathbf{D},\mathbf{B},\mathbf{v},\psi,q)\in \bm{H}_{0}(\mathbf{curl}) \times \bm{L}^{2}(\mathscr{D}) \times \bm{H}^{1}_{0} \times L^{2}(\mathscr{D}) \times H^{1}_{0}$.

The well-posedness of \eqref{e7.2.5} implies that \eqref{e7.4.3} has a unique solution. 
Next, we study uniform a priori estimates of the solution to \eqref{e7.4.3} with respect to $\lambda$. Define two bilinear forms
\begin{align*}
	\bar{a}(\mathbf{w},\mathbf{v}) =&\ (G\nabla\cdot \mathbf{w}, \nabla\cdot \mathbf{v}) + (G\nabla \mathbf{w}, \nabla \mathbf{v}), \quad \forall \, \mathbf{w},\mathbf{v}\in \bm{H}^{1}_{0},
	\\
	\bar{b}(\mathbf{v},\psi) =&\ (\psi, \nabla\cdot \mathbf{v}), \quad \forall \, \mathbf{v}\in \bm{H}^{1}_{0}, \, \psi\in L^{2}(\mathscr{D}).
\end{align*}
Then, $\bar{a}(\cdot,\cdot)$ is symmetric, coercive and bounded, which induces an equivalent norm $\|\cdot\|_{\bar{a}}=\sqrt{\bar{a}(\cdot,\cdot)}$ in $\bm{H}^{1}_{0}$, i.e., there exist two constants $0<c_{1}<c_{2}$ such that $c_{1}\|\cdot\|_{1}\leq \|\cdot\|_{\bar{a}}\leq c_{2}\|\cdot\|_{1}$. 
The bilinear form $\bar{b}(\cdot,\cdot)$ satisfies the continuous inf-sup condition \cite{GiRa1986,OyRu2016} 
\begin{equation}\label{e7.4.4}
	\sup_{\mathbf{0}\neq \mathbf{v}\in \bm{H}^{1}_{0}}\frac{\bar{b}(\mathbf{v},\psi)}{\|\mathbf{v}\|_{1}} \geq \delta\|\psi\|, \quad \forall \, \psi\in L^{2}(\mathscr{D}),
\end{equation}
with a constant $\delta>0$ only depending on $\mathscr{D}$.

\begin{lemma}\label{l7.4.1}
	Assume that (H1) and (H2) hold. Let $(\mathbf{E},\mathbf{H},\mathbf{u},\varphi,p)$ be the solution of \eqref{e7.4.3}. Then there exists a constant $C>0$ independent of $\lambda$ such that
	\begin{equation*}
		\|\mathbf{E}(t)\|^{2} + \|\mathbf{H}(t)\|^{2} + \|\mathbf{u}(t)\|_{1}^{2} + \|\varphi(t)\|^{2}+ \|p(t)\|^{2}  \leq C, \quad \forall \, t\in (0,T].
	\end{equation*}
\end{lemma}

\begin{proof}
	Differentiating the fourth equation of \eqref{e7.4.3} with respect to $t$, we have
	\begin{equation}\label{e7.4.5}
		\big(\frac{1}{\lambda}\frac{\partial}{\partial t}\varphi(t), \psi\big) 
		- \big(\frac{\alpha}{\lambda}\frac{\partial}{\partial t} p(t), \psi\big) 
		+ \bar{b}\big(\frac{\partial}{\partial t}\mathbf{u}(t), \psi\big) 
		= 0.
	\end{equation}
	Taking $(\mathbf{D},\mathbf{B},\mathbf{v},q)=\big(\mathbf{E}(t),\mathbf{H}(t),\frac{\partial}{\partial t}\mathbf{u}(t),p(t)\big)$ in the first three and last equations of \eqref{e7.4.3}, and $\psi=\varphi(t)$ in \eqref{e7.4.5}, respectively, then summing these up and applying the Cauchy--Schwarz and Young inequalities, we get
	\begin{equation}\label{e7.4.6}
		\begin{aligned}
			&\ \epsilon\frac{1}{2}\frac{\mathrm{d}}{\mathrm{d}t}\|\mathbf{E}(t)\|^{2} 
			+ \mu\frac{1}{2}\frac{\mathrm{d}}{\mathrm{d}t}\|\mathbf{H}(t)\|^{2} 
			+ \frac{1}{2}\frac{\mathrm{d}}{\mathrm{d}t}\|\mathbf{u}(t)\|_{\bar{a}}^{2}
			+ c_{0}\frac{1}{2}\frac{\mathrm{d}}{\mathrm{d}t}\|p(t)\|^{2} 
			\\
			&\ + \frac{1}{\lambda}\frac{1}{2}\frac{\mathrm{d}}{\mathrm{d}t}\|\alpha p(t)-\varphi(t)\|^{2}
			+ \sigma\|\mathbf{E}(t)\|^{2} + \kappa\|\nabla p(t)\|^{2}
			\\
			=&\ 2(L\mathbf{E}(t), \nabla p(t)) + (\mathbf{j}(t), \mathbf{E}(t)) + (g(t), p(t)) 
			\\
			\leq&\ \frac{L^{2}}{\kappa}\|\mathbf{E}(t)\|^{2} + \kappa\|\nabla p(t)\|^{2} 
			+ \frac{1}{4\sigma}\|\mathbf{j}(t)\|^{2} + \sigma\|\mathbf{E}(t)\|^{2} 
			\\
			&\ + \frac{1}{2}\|g(t)\|^{2} + \frac{1}{2}\|p(t)\|^{2}.
		\end{aligned}
	\end{equation}
	For any $t\in (0,T]$, integrating \eqref{e7.4.6} over $[0,t]$, we obtain
	\begin{equation}\label{e7.4.7}
		\begin{aligned}
			&\ \epsilon\|\mathbf{E}(t)\|^{2} + \mu\|\mathbf{H}(t)\|^{2} + c_{1}^{2}\|\mathbf{u}(t)\|_{1}^{2} 
			+ c_{0}\|p(t)\|^{2} 
			\\
			\leq&\ \epsilon\|\mathbf{E}(t)\|^{2} + \mu\|\mathbf{H}(t)\|^{2} + \|\mathbf{u}(t)\|_{\bar{a}}^{2} 
			+ c_{0}\|p(t)\|^{2} + \frac{1}{\lambda}\|\alpha p(t)-\varphi(t)\|^{2}
			\\
			\leq&\ \epsilon\|\mathbf{E}_{0}\|^{2} + \mu\|\mathbf{H}_{0}\|^{2} + \|\mathbf{u}_{0}\|_{\bar{a}}^{2} 
			+ c_{0}\|p_{0}\|^{2} + \frac{1}{\lambda}\|\alpha p_{0}-\varphi_{0}\|^{2}
			\\
			&\ + \frac{1}{2\sigma}\int_{0}^{t}{\|\mathbf{j}(\theta)\|^{2}}{\, \mathrm{d}\theta} 
			+ \int_{0}^{t}{\|g(\theta)\|^{2}}{\, \mathrm{d}\theta} 
			\\
			&\ + \frac{2L^{2}}{\kappa}\int_{0}^{t}{\|\mathbf{E}(\theta)\|^{2}}{\, \mathrm{d}\theta}
			+ \int_{0}^{t}{\|p(\theta)\|^{2}}{\, \mathrm{d}\theta}
			\\
			\leq&\ \epsilon \|\mathbf{E}_{0}\|^{2} + \mu \|\mathbf{H}_{0}\|^{2} + c_{2}^{2}\|\mathbf{u}_{0}\|_{1}^{2} 
			+ c_{0} \|p_{0}\|^{2} + \frac{2\alpha^{2}}{\lambda}\|p_{0}\|^{2} + \frac{2}{\lambda}\|\varphi_{0}\|^{2}
			\\
			&\ + \frac{1}{2\sigma}\int_{0}^{t}{\|\mathbf{j}(\theta)\|^{2}}{\, \mathrm{d}\theta} 
			+ \int_{0}^{t}{\|g(\theta)\|^{2}}{\, \mathrm{d}\theta} 
			\\
			&\ + \frac{2L^{2}}{\kappa}\int_{0}^{t}{\|\mathbf{E}(\theta)\|^{2}}{\, \mathrm{d}\theta}
			+ \int_{0}^{t}{\|p(\theta)\|^{2}}{\, \mathrm{d}\theta}.
		\end{aligned}
	\end{equation}
	Since $\mathbf{u}_{0}$ and $p_{0}$ satisfy the third equation of \eqref{e7.2.1}, we get
	\begin{equation*}
		(\lambda+G)\|\nabla\cdot\mathbf{u}_{0}\|^{2} + G\|\nabla\mathbf{u}_{0}\|^{2} 
		= (p_{0}, \alpha\nabla\cdot\mathbf{u}_{0}) 
		\leq \frac{\alpha}{2}\|p_{0}\|^{2} + \frac{\alpha}{2}\|\mathbf{u}_{0}\|_{1}^{2},
	\end{equation*}
	which together with \eqref{e7.4.2} leads to
	\begin{equation}\label{e7.4.8}
		\frac{1}{\lambda}\|\varphi_{0}\|^{2} 
		\leq \frac{2\alpha^{2}}{\lambda}\|p_{0}\|^{2} + 2\lambda\|\nabla\cdot\mathbf{u}_{0}\|^{2} 
		\leq \frac{2\alpha^{2}}{\lambda}\|p_{0}\|^{2} + \alpha\|p_{0}\|^{2} + \alpha\|\mathbf{u}_{0}\|_{1}^{2}.
	\end{equation}
	Noticing (H2) and applying the Gr\"{o}nwall lemma to \eqref{e7.4.7}, we obtain
	\begin{equation*}
		\|\mathbf{E}(t)\|^{2} + \|\mathbf{H}(t)\|^{2} 
		+ \|\mathbf{u}(t)\|_{1}^{2} + \|p(t)\|^{2}
		\leq C,
	\end{equation*}
	where $C>0$ is a constant independent of $\lambda$.
	By \eqref{e7.4.4} and the third equation of \eqref{e7.4.3}, we have
	\begin{align*}
		\delta\|\varphi(t)\| \leq&\ \sup_{\mathbf{0}\neq \mathbf{v}\in \bm{H}^{1}_{0}}\frac{\bar{b}(\mathbf{v},\varphi(t))}{\|\mathbf{v}\|_{1}} 
		= \sup_{\mathbf{0}\neq \mathbf{v}\in \bm{H}^{1}_{0}}\frac{\bar{a}(\mathbf{u}(t),\mathbf{v})}{\|\mathbf{v}\|_{1}} 
		\\
		\leq&\ \sup_{\mathbf{0}\neq \mathbf{v}\in \bm{H}^{1}_{0}}\frac{C\|\mathbf{u}(t)\|_{1}\|\mathbf{v}\|_{1}}{\|\mathbf{v}\|_{1}} 
		= C\|\mathbf{u}(t)\|_{1} \leq C,
	\end{align*}
	which completes the proof.
\end{proof}

\subsection{FEM to five-field model and uniform a priori estimates}\label{s7.4.2}
From now on, we always require integer $l\geq 2$. Construct a finite element space
\begin{equation*}
	\mathbb{G}_{h} = \{\varphi_{h}\in L^{2}(\mathscr{D}): \varphi_{h}|_{K}\in P_{l-1}(K), \ \forall\, K\in \mathcal{T}_{h}\}.
\end{equation*}
Let $\mathcal{G}_{h}: L^{2}(\mathscr{D})\to \mathbb{G}_{h}$ be $L^{2}$-orthogonal projection operator. 
Notice that the pair of spaces $(\mathbb{U}_{h},\mathbb{G}_{h})$ is composed of Taylor--Hood elements and satisfies the discrete inf-sup condition \cite{GiRa1986,OyRu2016}
\begin{equation}\label{e7.4.9}
	\sup_{\mathbf{0}\neq \mathbf{v}_{h}\in \mathbb{U}_{h}}\frac{\bar{b}(\mathbf{v}_{h},\psi_{h})}{\|\mathbf{v}_{h}\|_{1}} \geq \delta_{0}\|\psi_{h}\|, \quad \forall \, \psi_{h}\in \mathbb{G}_{h},
\end{equation}
with a constant $\delta_{0}>0$ independent of $h$.

\begin{lemma}\label{l7.4.2}
	(\cite[Lemma 11.18]{ErGu2021}) For any integer $\beta\geq 1$, there exists a constant $C>0$ independent of $h$ such that for any $\varphi\in H^{\beta}$,
	\begin{equation*}
		\|\varphi-\mathcal{G}_{h}\varphi\| \leq Ch^{\min\{l,\beta\}}\|\varphi\|_{\beta}.
	\end{equation*}
\end{lemma}

Define a Ritz projection operator $\mathscr{Q}_{h}:\bm{H}^{1}_{0}\to \mathbb{U}_{h}$ satisfying for any $\mathbf{u}\in \bm{H}^{1}_{0}$,
\begin{equation}\label{e7.4.10}
	\bar{a}(\mathscr{Q}_{h}\mathbf{u},\mathbf{v}_{h})  
	= \bar{a}(\mathbf{u},\mathbf{v}_{h}), \quad \forall \, \mathbf{v}_{h}\in \mathbb{U}_{h}.
\end{equation}
The coercivity and boundedness of $\bar{a}(\cdot,\cdot)$ implies the following lemma.

\begin{lemma}\label{l7.4.3}
	For any integer $\beta\geq 1$, there exists a constant $C>0$ independent of $h$ such that for any $\mathbf{u}\in \bm{H}^{\beta+1}\cap \bm{H}^{1}_{0}$,
	\begin{equation*}
		\|\mathscr{Q}_{h}\mathbf{u}-\mathbf{u}\|_{1} \leq Ch^{\min\{l,\beta\}}\|\mathbf{u}\|_{\beta+1}.
	\end{equation*}
\end{lemma}

We apply the backward Euler method and FEM to construct a fully discretized scheme to \eqref{e7.4.3}: find $(\mathbf{E}_{h}^{n},\mathbf{H}_{h}^{n},\mathbf{u}_{h}^{n},\varphi_{h}^{n},p_{h}^{n})\in \mathbb{E}_{h}\times \mathbb{H}_{h}\times \mathbb{U}_{h}\times \mathbb{G}_{h}\times \mathbb{P}_{h}$ such that
\begin{equation}\label{e7.4.11}
	(\mathbf{E}_{h}^{0},\mathbf{H}_{h}^{0},\mathbf{u}_{h}^{0},\varphi_{h}^{0},p_{h}^{0})
	= (\mathfrak{P}_{h}\mathbf{E}_{0},\mathcal{P}_{h}\mathbf{H}_{0},\mathscr{P}_{h}\mathbf{u}_{0},\mathcal{G}_{h}\varphi_{0},P_{h}p_{0}),
\end{equation}
and
\begin{equation}\label{e7.4.12}
	\begin{aligned}
		& (\epsilon \bar\partial_{t}\mathbf{E}_{h}^{n},\mathbf{D}_{h}) 
		+ (\sigma \mathbf{E}_{h}^{n},\mathbf{D}_{h}) 
		- (\mathbf{H}_{h}^{n},\nabla\times \mathbf{D}_{h}) 
		- (L\nabla p_{h}^{n},\mathbf{D}_{h}) 
		= (\mathbf{j}(t_{n}),\mathbf{D}_{h}),
		\\
		& (\mu \bar\partial_{t}\mathbf{H}_{h}^{n},\mathbf{B}_{h}) 
		+ (\nabla\times \mathbf{E}_{h}^{n},\mathbf{B}_{h}) = 0,
		\\
		& \bar{a}(\mathbf{u}_{h}^{n}, \mathbf{v}_{h}) 
		- \bar{b}(\mathbf{v}_{h}, \varphi_{h}^{n})
		= 0,
		\\
		& \big(\frac{1}{\lambda}\varphi_{h}^{n}, \psi_{h}\big) 
		- \big(\frac{\alpha}{\lambda} p_{h}^{n}, \psi_{h}\big) 
		+ \bar{b}(\mathbf{u}_{h}^{n}, \psi_{h}) 
		= 0,
		\\
		& \big((c_{0}+\frac{\alpha^{2}}{\lambda})\bar\partial_{t}p_{h}^{n}, q_{h}\big)
		- \big(\frac{\alpha}{\lambda}\bar\partial_{t}\varphi_{h}^{n}, q_{h}\big) 
		+ (\kappa\nabla p_{h}^{n},\nabla q_{h}) 
		- (L\mathbf{E}_{h}^{n},\nabla q_{h}) 
		\\
		&\ \qquad\qquad = (g(t_{n}),q_{h}),
	\end{aligned}
\end{equation}
for all $n=1,2,\ldots,N$ and $(\mathbf{D}_{h},\mathbf{B}_{h},\mathbf{v}_{h},\psi_{h},q_{h})\in \mathbb{E}_{h}\times \mathbb{H}_{h}\times \mathbb{U}_{h}\times \mathbb{G}_{h}\times \mathbb{P}_{h}$.

In the next theorem, we study the existence, uniqueness and uniform stability of the solution to \eqref{e7.4.12} and \eqref{e7.4.11}.

\begin{theorem}\label{t7.4.1}
	Assume that (H1) and (H2) hold. The fully discretized equations \eqref{e7.4.12} and \eqref{e7.4.11} admit a unique solution $(\mathbf{E}_{h}^{n},\mathbf{H}_{h}^{n},\mathbf{u}_{h}^{n},\varphi_{h}^{n},p_{h}^{n})$. Furthermore, there exists a constant $C>0$ independent of $\tau$, $h$ and $\lambda$ such that for all $1\leq n\leq N$,
	\begin{equation*}
		\|\mathbf{E}_{h}^{n}\|^{2} + \|\mathbf{H}_{h}^{n}\|^{2} 
		+ \|\mathbf{u}_{h}^{n}\|_{1}^{2} + \|\varphi_{h}^{n}\|^{2} 
		+ \|p_{h}^{n}\|^{2} 
		\leq C.
	\end{equation*}
\end{theorem}

\begin{proof}
	Notice that the number of linear equations \eqref{e7.4.12} equals the number of unknowns, and hence the uniqueness implies the existence.
	
	According to the fourth equation of \eqref{e7.4.12}, we have
	\begin{equation}\label{e7.4.13}
		\big(\frac{1}{\lambda}\bar\partial_{t}\varphi_{h}^{n}, \psi_{h}\big) 
		- \big(\frac{\alpha}{\lambda}\bar\partial_{t} p_{h}^{n}, \psi_{h}\big) 
		+ \bar{b}(\bar\partial_{t}\mathbf{u}_{h}^{n}, \psi_{h}) 
		= 0, \quad \forall\, \psi_{h}\in \mathbb{G}_{h}.
	\end{equation}
	Taking $(\mathbf{D}_{h},\mathbf{B}_{h},\mathbf{v}_{h},q_{h})=(\mathbf{E}_{h}^{n},\mathbf{H}_{h}^{n},\bar\partial_{t}\mathbf{u}_{h}^{n},p_{h}^{n})$ in the first three and last equations of \eqref{e7.4.12}, and $\psi_{h}=\varphi_{h}^{n}$ in \eqref{e7.4.13}, respectively, then summing these up and applying the Cauchy--Schwarz and Young inequalities, we get
	\begin{equation}\label{e7.4.14}
		\begin{aligned}
			&\ (\epsilon \bar\partial_{t}\mathbf{E}_{h}^{n},\mathbf{E}_{h}^{n}) 
			+ (\mu \bar\partial_{t}\mathbf{H}_{h}^{n},\mathbf{H}_{h}^{n}) 
			+ \bar{a}(\mathbf{u}_{h}^{n}, \bar\partial_{t}\mathbf{u}_{h}^{n}) 
			+ (c_{0}\bar\partial_{t}p_{h}^{n}, p_{h}^{n})
			\\
			&\ + \big(\frac{1}{\lambda}\bar\partial_{t}(\alpha p_{h}^{n}-\varphi_{h}^{n}),\alpha p_{h}^{n}-\varphi_{h}^{n}\big) 
			+ \sigma \|\mathbf{E}_{h}^{n}\|^{2} 
			+ \kappa \|\nabla p_{h}^{n}\|^{2}
			\\
			=&\  2(L \mathbf{E}_{h}^{n},\nabla p_{h}^{n}) 
			+ (\mathbf{j}(t_{n}),\mathbf{E}_{h}^{n}) + (g(t_{n}),p_{h}^{n}) 
			\\
			\leq&\ \frac{L^{2}}{\kappa}\|\mathbf{E}_{h}^{n}\|^{2} + \kappa\|\nabla p_{h}^{n}\|^{2} 
			+ \frac{1}{4\sigma}\|\mathbf{j}(t_{n})\|^{2} + \sigma\|\mathbf{E}_{h}^{n}\|^{2} 
			\\
			&\ + \frac{1}{2}\|g(t_{n})\|^{2} + \frac{1}{2}\|p_{h}^{n}\|^{2}.
		\end{aligned}
	\end{equation}
	In a similar way as proving \eqref{e7.3.8}, we obtain 
	\begin{equation}\label{e7.4.15}
		\begin{aligned}
			\tau\bar{a}(\mathbf{u}_{h}^{n}, \bar\partial_{t}\mathbf{u}_{h}^{n}) 
			\geq&\ \frac{1}{2}\|\mathbf{u}_{h}^{n}\|_{\bar{a}}^{2} - \frac{1}{2}\|\mathbf{u}_{h}^{n-1}\|_{\bar{a}}^{2}, 
			\\
			\tau(\bar\partial_{t}(\alpha p_{h}^{n}-\varphi_{h}^{n}),\alpha p_{h}^{n}-\varphi_{h}^{n}) 
			\geq&\ \frac{1}{2}\|\alpha p_{h}^{n}-\varphi_{h}^{n}\|^{2} - \frac{1}{2}\|\alpha p_{h}^{n-1}-\varphi_{h}^{n-1}\|^{2}.
		\end{aligned}
	\end{equation}
	According to \eqref{e7.4.14}, \eqref{e7.4.15}, \eqref{e7.3.8} and \eqref{e7.3.9}, we have
	\begin{align*}
		&\ \epsilon\|\mathbf{E}_{h}^{n}\|^{2} 
		+ \mu\|\mathbf{H}_{h}^{n}\|^{2}
		+ \|\mathbf{u}_{h}^{n}\|_{\bar{a}}^{2}
		+ c_{0}\|p_{h}^{n}\|^{2} 
		+ \frac{1}{\lambda}\|\alpha p_{h}^{n}-\varphi_{h}^{n}\|^{2}
		\\
		\leq&\ \epsilon\|\mathbf{E}_{h}^{n-1}\|^{2} 
		+ \mu\|\mathbf{H}_{h}^{n-1}\|^{2}
		+ \|\mathbf{u}_{h}^{n-1}\|_{\bar{a}}^{2}
		+ c_{0}\|p_{h}^{n-1}\|^{2} 
		+ \frac{1}{\lambda}\|\alpha p_{h}^{n-1}-\varphi_{h}^{n-1}\|^{2}
		\\
		&\ + \frac{1}{2\sigma}\tau\|\mathbf{j}(t_{n})\|^{2} 
		+ \tau\|g(t_{n})\|^{2} 
		+ \frac{2L^{2}}{\kappa}\tau\|\mathbf{E}_{h}^{n}\|^{2} 
		+ \tau\|p_{h}^{n}\|^{2}.
	\end{align*}
	Inductively for $n\geq 1$ in the above inequality, we get
	\begin{equation}\label{e7.4.16}
		\begin{aligned}
			&\ \epsilon\|\mathbf{E}_{h}^{n}\|^{2} 
			+ \mu\|\mathbf{H}_{h}^{n}\|^{2}
			+ \|\mathbf{u}_{h}^{n}\|_{\bar{a}}^{2}
			+ c_{0}\|p_{h}^{n}\|^{2} 
			+ \frac{1}{\lambda}\|\alpha p_{h}^{n}-\varphi_{h}^{n}\|^{2}
			\\
			\leq&\ \epsilon\|\mathbf{E}_{h}^{0}\|^{2} 
			+ \mu\|\mathbf{H}_{h}^{0}\|^{2}
			+ \|\mathbf{u}_{h}^{0}\|_{\bar{a}}^{2}
			+ c_{0}\|p_{h}^{0}\|^{2} 
			+ \frac{1}{\lambda}\|\alpha p_{h}^{0}-\varphi_{h}^{0}\|^{2}
			\\
			&\ + \frac{1}{2\sigma}\sum_{i=1}^{n}\tau\|\mathbf{j}(t_{i})\|^{2} 
			+ \sum_{i=1}^{n}\tau\|g(t_{i})\|^{2}
			+ \frac{2L^{2}}{\kappa}\sum_{i=1}^{n}\tau\|\mathbf{E}_{h}^{i}\|^{2} 
			+ \sum_{i=1}^{n}\tau\|p_{h}^{i}\|^{2}.
		\end{aligned}
	\end{equation}
	By virtue of \eqref{e7.4.11}, \eqref{e7.3.11} and \eqref{e7.4.8}, we have
	\begin{align*}
		\|\mathbf{u}_{h}^{0}\|_{\bar{a}} 
		\leq&\ c_{2}\|\mathscr{P}_{h}\mathbf{u}_{0}\|_{1}
		\leq C\|\mathbf{u}_{0}\|_{1},
		\\
		\frac{1}{\lambda}\|\alpha p_{h}^{0}-\varphi_{h}^{0}\|^{2} 
		\leq&\ \frac{2\alpha^{2}}{\lambda}\|p_{h}^{0}\|^{2} + \frac{2}{\lambda}\|\varphi_{h}^{0}\|^{2} 
		\leq \frac{6\alpha^{2}}{\lambda}\|p_{0}\|^{2} + 2\alpha\|p_{0}\|^{2} + 2\alpha\|\mathbf{u}_{0}\|_{1}^{2},
	\end{align*}
	where $C>0$ is a constant independent of $\tau$, $h$ and $\lambda$. Noticing \eqref{e7.3.11} and (H2), and then applying the discrete Gr\"{o}nwall lemma to \eqref{e7.4.16}, we obtain
	\begin{equation*}
		\|\mathbf{E}_{h}^{n}\|^{2} 
		+ \|\mathbf{H}_{h}^{n}\|^{2}
		+ \|\mathbf{u}_{h}^{n}\|_{\bar{a}}^{2}
		+ \|p_{h}^{n}\|^{2} 
		\leq C e^{C\sum_{i=1}^{n}\tau} = C,
	\end{equation*}
	where $C>0$ is a constant independent of $\tau$, $h$ and $\lambda$. Thus, $\|\mathbf{u}_{h}^{n}\|_{1} \leq \frac{1}{c_{1}}\|\mathbf{u}_{h}^{n}\|_{\bar{a}} \leq C$. By means of \eqref{e7.4.9} and the third equation of \eqref{e7.4.12}, we get
	\begin{align*}
		\delta_{0}\|\varphi_{h}^{n}\| \leq&\ \sup_{\mathbf{0}\neq \mathbf{v}_{h}\in \mathbb{U}_{h}}\frac{\bar{b}(\mathbf{v}_{h},\varphi_{h}^{n})}{\|\mathbf{v}_{h}\|_{1}} 
		= \sup_{\mathbf{0}\neq \mathbf{v}_{h}\in \mathbb{U}_{h}}\frac{\bar{a}(\mathbf{u}_{h}^{n}, \mathbf{v}_{h})}{\|\mathbf{v}_{h}\|_{1}} 
		\\
		\leq&\ \sup_{\mathbf{0}\neq \mathbf{v}_{h}\in \mathbb{U}_{h}}\frac{C\|\mathbf{u}_{h}^{n}\|_{1}\|\mathbf{v}_{h}\|_{1}}{\|\mathbf{v}_{h}\|_{1}} 
		= C\|\mathbf{u}_{h}^{n}\|_{1} \leq C.
	\end{align*}
	
	Now, we study the uniqueness of solution to \eqref{e7.4.12} and \eqref{e7.4.11}. For contradiction, assume that $\tilde{\mathbf{U}}_{i,h}^{n}:= (\mathbf{E}_{i,h}^{n},\mathbf{H}_{i,h}^{n},\mathbf{u}_{i,h}^{n},\varphi_{i,h}^{n},p_{i,h}^{n})$ $(i=1,2;\ n=1,2,\ldots,N)$ are two solutions of \eqref{e7.4.12} and \eqref{e7.4.11}. Then $\tilde{\mathbf{U}}_{1,h}^{n}-\tilde{\mathbf{U}}_{2,h}^{n}$ satisfies homogeneous equations with homogeneous initial and boundary conditions. Thus, we obtain for all $1\leq n\leq N$,
	\begin{align*}
		&\ \|\mathbf{E}_{1,h}^{n}-\mathbf{E}_{2,h}^{n}\|^{2} 
		+ \|\mathbf{H}_{1,h}^{n}-\mathbf{H}_{2,h}^{n}\|^{2} 
		+ \|\mathbf{u}_{1,h}^{n}-\mathbf{u}_{2,h}^{n}\|_{1}^{2} 
		\\
		&\ + \|\varphi_{1,h}^{n}-\varphi_{2,h}^{n}\|^{2} 
		+ \|p_{1,h}^{n}-p_{2,h}^{n}\|^{2} 
		\leq 0,
	\end{align*}
	which implies $\tilde{\mathbf{U}}_{1,h}^{n}=\tilde{\mathbf{U}}_{2,h}^{n}$.
	Then, the proof is completed.
\end{proof}

\subsection{High-order locking-free error estimates}\label{s7.4.3}
In this part, we investigate the uniform error estimates between the exact and numerical solutions to \eqref{e7.4.3}  with respect to $\lambda$. We further assume  
\begin{itemize}
	\item [\bf{(H4)}] For any integer $\beta\geq 1$, there exists a constant $C>0$ independent of $\lambda$ such that
	\begin{equation*}
		\|\varphi(t)\|_{\beta}^{2} + \int_{0}^{t}{\|\frac{\partial}{\partial t} \varphi(\theta)\|_{\beta}^{2}}{\, \mathrm{d}\theta} 
		+ \int_{0}^{t}{\|\frac{\partial^{2}}{\partial t^{2}} \varphi(\theta)\|^{2}}{\, \mathrm{d}\theta} \leq C, \quad \forall \, t\in (0,T].
	\end{equation*}
\end{itemize}

Define error functions for $0\leq n\leq N$,
\begin{align*}
	\mathbf{u}(t_{n})-\mathbf{u}_{h}^{n} =&\ (\mathbf{u}(t_{n})-\mathscr{Q}_{h}\mathbf{u}(t_{n})) + (\mathscr{Q}_{h}\mathbf{u}(t_{n})-\mathbf{u}_{h}^{n}) =: \tilde{\boldsymbol{\rho}}(t_{n}) + \tilde{\boldsymbol{\xi}}_{h}^{n}, 
	\\
	\varphi(t_{n})-\varphi_{h}^{n} =&\ (\varphi(t_{n})-\mathcal{G}_{h}\varphi(t_{n})) + (\mathcal{G}_{h}\varphi(t_{n})-\varphi_{h}^{n}) =: \pi(t_{n}) + \zeta_{h}^{n}.
\end{align*}
From Lemmas \ref{l7.4.2} and \ref{l7.4.3}, it follows that there exists a constant $C>0$ independent of $\tau$, $h$ and $\lambda$ such that for all $1\leq n\leq N$,
\begin{equation}\label{e7.4.17}
	\begin{aligned}
		\|\tilde{\boldsymbol{\rho}}(t_{n})\|_{1} =&\ \|\mathbf{u}(t_{n})-\mathscr{Q}_{h}\mathbf{u}(t_{n})\|_{1} 
		\leq Ch^{\min\{l,\beta\}}\|\mathbf{u}(t_{n})\|_{\beta+1} \leq Ch^{\min\{l,\beta\}}, 
		\\
		\|\pi(t_{n})\| =&\ \|\varphi(t_{n})-\mathcal{G}_{h}\varphi(t_{n})\| 
		\leq Ch^{\min\{l,\beta\}}\|\varphi(t_{n})\|_{\beta} \leq Ch^{\min\{l,\beta\}}.
	\end{aligned}
\end{equation}

Now, we are ready to derive error estimates between the exact and numerical solutions to \eqref{e7.4.3}.

\begin{theorem}\label{t7.4.2}
	Assume (H1)--(H4). Let $(\mathbf{E},\mathbf{H},\mathbf{u},\varphi,p)$ and $(\mathbf{E}_{h}^{n},\mathbf{H}_{h}^{n},\mathbf{u}_{h}^{n},\varphi_{h}^{n},p_{h}^{n})$ be the exact and numerical solutions to \eqref{e7.4.3}, respectively. Then there exists a constant $C>0$ independent of $\tau$, $h$ and $\lambda$ such that for all $1\leq n\leq N$,
	\begin{align*}
		&\ \|\mathbf{E}(t_{n})-\mathbf{E}_{h}^{n}\|^{2} 
		+ \|\mathbf{H}(t_{n})-\mathbf{H}_{h}^{n}\|^{2} 
		+ \|\mathbf{u}(t_{n})-\mathbf{u}_{h}^{n}\|_{1}^{2} 
		\\
		&\ + \|\varphi(t_{n})-\varphi_{h}^{n}\|^{2} 
		+ \|p(t_{n})-p_{h}^{n}\|^{2} 
		\leq C(\tau^{2}+h^{2\min\{l,\beta\}}).
	\end{align*}
\end{theorem}

\begin{proof}
	Noticing \eqref{e7.4.10} and utilizing the third equations of \eqref{e7.4.3} and of \eqref{e7.4.12}, respectively, we obtain for any $\mathbf{v}_{h}\in \mathbb{U}_{h}$,
	\begin{equation}\label{e7.4.18}
		\begin{aligned}
			&\ \bar{a}(\tilde{\boldsymbol{\xi}}_{h}^{n},\mathbf{v}_{h}) - \bar{b}(\mathbf{v}_{h},\zeta_{h}^{n}) 
			\\
			=&\ \bar{a}(\mathscr{Q}_{h}\mathbf{u}(t_{n}),\mathbf{v}_{h}) 
			- \bar{b}(\mathbf{v}_{h},\mathcal{G}_{h}\varphi(t_{n})) 
			- \bar{a}(\mathbf{u}_{h}^{n},\mathbf{v}_{h}) 
			+ \bar{b}(\mathbf{v}_{h},\varphi_{h}^{n})
			\\
			=&\ \bar{a}(\mathbf{u}(t_{n}),\mathbf{v}_{h}) 
			- \bar{b}(\mathbf{v}_{h},\varphi(t_{n})) 
			- \bar{a}(\mathbf{u}_{h}^{n},\mathbf{v}_{h}) 
			+ \bar{b}(\mathbf{v}_{h},\varphi_{h}^{n})
			=0,
		\end{aligned}
	\end{equation}
	where we have used $\bar{b}(\mathbf{v}_{h},\mathcal{G}_{h}\varphi(t_{n}))=\bar{b}(\mathbf{v}_{h},\varphi(t_{n}))$ due to $\nabla\cdot \mathbf{v}_{h}|_{K} \in \mathbb{G}_{h}|_{K}$. 
	According to the fourth equations of \eqref{e7.4.3} and of \eqref{e7.4.12}, respectively, we get for any $\psi_{h}\in \mathbb{G}_{h}$,
	\begin{align*}
		&\ \big(\frac{1}{\lambda}\zeta_{h}^{n}, \psi_{h}\big) 
		- \big(\frac{\alpha}{\lambda}\eta_{h}^{n}, \psi_{h}\big) 
		+ \bar{b}(\tilde{\boldsymbol{\xi}}_{h}^{n}, \psi_{h}) 
		\\
		=&\ \big(\frac{1}{\lambda}\mathcal{G}_{h}\varphi(t_{n}), \psi_{h}\big) 
		- \big(\frac{\alpha}{\lambda}\mathcal{R}_{h}p(t_{n}), \psi_{h}\big) 
		+ \bar{b}(\mathscr{Q}_{h}\mathbf{u}(t_{n}), \psi_{h}) 
		\\
		&\ - \big(\frac{1}{\lambda}\varphi_{h}^{n}, \psi_{h}\big) 
		+ \big(\frac{\alpha}{\lambda} p_{h}^{n}, \psi_{h}\big) 
		- \bar{b}(\mathbf{u}_{h}^{n}, \psi_{h}) 
		\\
		=&\ \big(\frac{1}{\lambda}\varphi(t_{n}), \psi_{h}\big) 
		- \big(\frac{\alpha}{\lambda} \mathcal{R}_{h}p(t_{n}), \psi_{h}\big) 
		+ \bar{b}(\mathscr{Q}_{h}\mathbf{u}(t_{n}), \psi_{h}) 
		\\
		&\ -\big(\frac{1}{\lambda}\varphi(t_{n}), \psi_{h}\big) 
		+ \big(\frac{\alpha}{\lambda}p(t_{n}), \psi_{h}\big) 
		- \bar{b}(\mathbf{u}(t_{n}), \psi_{h}) 
		\\
		=&\ \big(\frac{\alpha}{\lambda}\gamma(t_{n}), \psi_{h}\big) 
		- \bar{b}(\tilde{\boldsymbol{\rho}}(t_{n}), \psi_{h}),
	\end{align*}
	which yields
	\begin{equation}\label{e7.4.19}
		\begin{aligned}
			&\ \big(\frac{1}{\lambda}\bar\partial_{t}\zeta_{h}^{n}, \psi_{h}\big) 
			- \big(\frac{\alpha}{\lambda}\bar\partial_{t}\eta_{h}^{n}, \psi_{h}\big) 
			+ \bar{b}(\bar\partial_{t}\tilde{\boldsymbol{\xi}}_{h}^{n}, \psi_{h}) 
			\\
			=&\ \big(\frac{\alpha}{\lambda}\bar\partial_{t}\gamma(t_{n}), \psi_{h}\big) 
			- \bar{b}(\bar\partial_{t}\tilde{\boldsymbol{\rho}}(t_{n}), \psi_{h}).
		\end{aligned}
	\end{equation}
	Noticing \eqref{e7.3.2} and applying the last equations of \eqref{e7.4.3} and of \eqref{e7.4.12}, respectively, we have for any $q_{h}\in \mathbb{P}_{h}$,
	\begin{equation}\label{e7.4.20}
		\begin{aligned}
			&\ \big((c_{0}+\frac{\alpha^{2}}{\lambda})\bar\partial_{t}\eta_{h}^{n}, q_{h}\big)
			- \big(\frac{\alpha}{\lambda}\bar\partial_{t}\zeta_{h}^{n}, q_{h}\big) 
			+ (\kappa\nabla \eta_{h}^{n},\nabla q_{h}) 
			\\
			=&\ \big((c_{0}+\frac{\alpha^{2}}{\lambda})\bar\partial_{t}\mathcal{R}_{h}p(t_{n}), q_{h}\big)
			- \big(\frac{\alpha}{\lambda}\bar\partial_{t}\mathcal{G}_{h}\varphi(t_{n}), q_{h}\big) 
			+ (\kappa\nabla \mathcal{R}_{h}p(t_{n}),\nabla q_{h}) 
			\\
			&\ - \big((c_{0}+\frac{\alpha^{2}}{\lambda})\bar\partial_{t} p_{h}^{n}, q_{h}\big)
			+ \big(\frac{\alpha}{\lambda}\bar\partial_{t}\varphi_{h}^{n}, q_{h}\big) 
			- (\kappa\nabla p_{h}^{n},\nabla q_{h}) 
			\\
			=&\ \big((c_{0}\!+\!\frac{\alpha^{2}}{\lambda})\bar\partial_{t}(\mathcal{R}_{h}p(t_{n})\!-\!p(t_{n})),q_{h}\big) 
			\!+\! \big((c_{0}\!+\!\frac{\alpha^{2}}{\lambda})(\bar\partial_{t}p(t_{n})\!-\!\frac{\partial}{\partial t}p(t_{n})),q_{h}\big) 
			\\
			&\ + \big((c_{0}+\frac{\alpha^{2}}{\lambda})\frac{\partial}{\partial t}p(t_{n}),q_{h}\big) 
			- \big(\frac{\alpha}{\lambda}\bar\partial_{t}(\mathcal{G}_{h}\varphi(t_{n})-\varphi(t_{n})),q_{h}\big) 
			\\
			&\ - \big(\frac{\alpha}{\lambda}(\bar\partial_{t}\varphi(t_{n})-\frac{\partial}{\partial t}\varphi(t_{n})),q_{h}\big) 
			- \big(\frac{\alpha}{\lambda}\frac{\partial}{\partial t}\varphi(t_{n}),q_{h}\big) 
			+ (\kappa\nabla p(t_{n}),\nabla q_{h}) 
			\\
			&\ - \big((c_{0}+\frac{\alpha^{2}}{\lambda})\bar\partial_{t} p_{h}^{n}, q_{h}\big)
			+ \big(\frac{\alpha}{\lambda}\bar\partial_{t}\varphi_{h}^{n}, q_{h}\big) 
			- (\kappa\nabla p_{h}^{n},\nabla q_{h}) 
			\\
			=&\ -\big((c_{0}+\frac{\alpha^{2}}{\lambda})\bar\partial_{t}\gamma(t_{n}),q_{h}\big) 
			+ \big((c_{0}+\frac{\alpha^{2}}{\lambda})e_{3}(t_{n}),q_{h}\big) 
			+ \big(\frac{\alpha}{\lambda}\bar\partial_{t}\pi(t_{n}),q_{h}\big) 
			\\
			&\ - \big(\frac{\alpha}{\lambda}(\bar\partial_{t}\varphi(t_{n})-\frac{\partial}{\partial t}\varphi(t_{n})),q_{h}\big) 
			+ (L\Xi(t_{n}),\nabla q_{h}) 
			+ (L\Pi_{h}^{n},\nabla q_{h}).
		\end{aligned}
	\end{equation}
	For simplifying notations, we define 
	\begin{equation*}
		e_{5}(t_{n}) = \bar\partial_{t}\varphi(t_{n})-\frac{\partial}{\partial t}\varphi(t_{n}).
	\end{equation*}
	Taking $\mathbf{D}_{h}=\Pi_{h}^{n}$, $\mathbf{B}_{h}=\Theta_{h}^{n}$, $\mathbf{v}_{h}=\bar\partial_{t}\tilde{\boldsymbol{\xi}}_{h}^{n}$, $\psi_{h}=\zeta_{h}^{n}$, $q_{h}=\eta_{h}^{n}$ in \eqref{e7.3.13}, \eqref{e7.3.14} and \eqref{e7.4.18}--\eqref{e7.4.20}, respectively, and summing these up, we obtain an error equation
	\begin{equation}\label{e7.4.21}
		\begin{aligned}
			&\ (\epsilon\bar\partial_{t}\Pi_{h}^{n},\Pi_{h}^{n}) 
			+ (\mu\bar\partial_{t}\Theta_{h}^{n},\Theta_{h}^{n}) 
			+ \bar{a}(\tilde{\boldsymbol{\xi}}_{h}^{n},\bar\partial_{t}\tilde{\boldsymbol{\xi}}_{h}^{n}) 
			+ (c_{0}\bar\partial_{t}\eta_{h}^{n}, \eta_{h}^{n}) 
			\\
			&\ + \big(\frac{1}{\lambda}\bar\partial_{t}(\alpha\eta_{h}^{n}-\zeta_{h}^{n}), \alpha\eta_{h}^{n}-\zeta_{h}^{n}\big) 
			+ (\kappa\nabla \eta_{h}^{n},\nabla \eta_{h}^{n}) 
			\\
			=&\ -(\epsilon\bar\partial_{t}\Xi(t_{n}),\Pi_{h}^{n}) 
			+ (\epsilon\bm{e}_{1}(t_{n}),\Pi_{h}^{n}) 
			- (\sigma \Xi(t_{n}),\Pi_{h}^{n}) 
			- (\sigma \Pi_{h}^{n},\Pi_{h}^{n}) 
			\\
			&\ + (L\nabla \gamma(t_{n}),\Pi_{h}^{n}) 
			+ 2(L\nabla \eta_{h}^{n},\Pi_{h}^{n})
			-(\mu\bar\partial_{t}\Lambda(t_{n}),\Theta_{h}^{n})  
			+ (\mu\bm{e}_{2}(t_{n}),\Theta_{h}^{n})
			\\
			&\ - (\nabla\times \Xi(t_{n}),\Theta_{h}^{n}) 
			+ \big(\frac{\alpha}{\lambda}\bar\partial_{t}\gamma(t_{n}), \zeta_{h}^{n}\big) 
			- \bar{b}(\bar\partial_{t}\tilde{\boldsymbol{\rho}}(t_{n}), \zeta_{h}^{n})
			\\
			&\ -\big((c_{0}+\frac{\alpha^{2}}{\lambda})\bar\partial_{t}\gamma(t_{n}),\eta_{h}^{n}\big) 
			+ \big((c_{0}+\frac{\alpha^{2}}{\lambda})e_{3}(t_{n}),\eta_{h}^{n}\big) 
			+ \big(\frac{\alpha}{\lambda}\bar\partial_{t}\pi(t_{n}),\eta_{h}^{n}\big)  
			\\
			&\ - \big(\frac{\alpha}{\lambda}e_{5}(t_{n}),\eta_{h}^{n}\big) 
			+ (L\Xi(t_{n}),\nabla \eta_{h}^{n}).
		\end{aligned}
	\end{equation}
	In a similar way as proving \eqref{e7.3.8}, we conclude that
	\begin{equation}\label{e7.4.22}
		\begin{aligned}
			\tau\bar{a}(\tilde{\boldsymbol{\xi}}_{h}^{n},\bar\partial_{t}\tilde{\boldsymbol{\xi}}_{h}^{n})
			\geq&\ \frac{1}{2}\|\tilde{\boldsymbol{\xi}}_{h}^{n}\|_{\bar{a}}^{2} - \frac{1}{2}\|\tilde{\boldsymbol{\xi}}_{h}^{n-1}\|_{\bar{a}}^{2},
			\\
			\tau(\bar\partial_{t}(\alpha\eta_{h}^{n}-\zeta_{h}^{n}), \alpha\eta_{h}^{n}-\zeta_{h}^{n}) 
			\geq&\ \frac{1}{2}\|\alpha\eta_{h}^{n}-\zeta_{h}^{n}\|^{2} - \frac{1}{2}\|\alpha\eta_{h}^{n-1}-\zeta_{h}^{n-1}\|^{2}.
		\end{aligned}
	\end{equation}
	From \eqref{e7.3.18}, \eqref{e7.4.21}, \eqref{e7.4.22}, the Cauchy--Schwarz and Young inequalities, it follows that
	\begin{equation}\label{e7.4.23}
		\begin{aligned}
			&\ \frac{\epsilon}{2}\big(\|\Pi_{h}^{n}\|^{2}-\|\Pi_{h}^{n-1}\|^{2}\big) 
			+ \frac{\mu}{2}\big(\|\Theta_{h}^{n}\|^{2}-\|\Theta_{h}^{n-1}\|^{2}\big) 
			\\
			&\ + \frac{1}{2}\big(\|\tilde{\boldsymbol{\xi}}_{h}^{n}\|_{\bar{a}}^{2}-\|\tilde{\boldsymbol{\xi}}_{h}^{n-1}\|_{\bar{a}}^{2}\big)
			+ \frac{c_{0}}{2}\big(\|\eta_{h}^{n}\|^{2}-\|\eta_{h}^{n-1}\|^{2}\big) 
			\\
			&\ + \frac{1}{2\lambda}\big(\|\alpha\eta_{h}^{n}-\zeta_{h}^{n}\|^{2}-\|\alpha\eta_{h}^{n-1}-\zeta_{h}^{n-1}\|^{2}\big) 
			+ \kappa\tau\|\nabla \eta_{h}^{n}\|^{2}
			\\
			\leq&\ \frac{\epsilon}{2}\tau\|\bar\partial_{t}\Xi(t_{n})\|^{2} 
			+ \frac{\epsilon}{2}\tau\|\Pi_{h}^{n}\|^{2} 
			+ \frac{\epsilon}{2}\tau\|\bm{e}_{1}(t_{n})\|^{2} 
			+ \frac{\epsilon}{2}\tau\|\Pi_{h}^{n}\|^{2} 
			\\
			&\ + \frac{\sigma}{2}\tau\|\Xi(t_{n})\|^{2} 
			+ \frac{\sigma}{2}\tau\|\Pi_{h}^{n}\|^{2} 
			+ \sigma\tau\|\Pi_{h}^{n}\|^{2} 
			+ \frac{L}{2}\tau\|\nabla \gamma(t_{n})\|^{2} 
			\\
			&\ + \frac{L}{2}\tau\|\Pi_{h}^{n}\|^{2} 
			+ \frac{\kappa}{2}\tau\|\nabla \eta_{h}^{n}\|^{2} 
			+ \frac{2L^{2}}{\kappa}\tau\|\Pi_{h}^{n}\|^{2} 
			+ \frac{\mu}{2}\tau\|\bar\partial_{t}\Lambda(t_{n})\|^{2} 
			\\
			&\ + \frac{\mu}{2}\tau\|\Theta_{h}^{n}\|^{2} 
			+ \frac{\mu}{2}\tau\|\bm{e}_{2}(t_{n})\|^{2} 
			+ \frac{\mu}{2}\tau\|\Theta_{h}^{n}\|^{2} 
			+ \frac{1}{2}\tau\|\nabla\times \Xi(t_{n})\|^{2} 
			\\
			&\ + \frac{1}{2}\tau\|\Theta_{h}^{n}\|^{2} 
			+ \frac{\alpha^{2}}{2\lambda}\tau\|\bar\partial_{t}\gamma(t_{n})\|^{2} 
			+ \frac{1}{2\lambda}\tau\|\zeta_{h}^{n}\|^{2} 
			+ \frac{1}{2}\tau\|\bar\partial_{t}\tilde{\boldsymbol{\rho}}(t_{n})\|_{1}^{2} 
			\\
			&\ + \frac{1}{2}\tau\|\zeta_{h}^{n}\|^{2} 
			+ (\frac{c_{0}}{2}+\frac{\alpha^{2}}{2\lambda})\tau\|\bar\partial_{t}\gamma(t_{n})\|^{2} 
			+ (\frac{c_{0}}{2}+\frac{\alpha^{2}}{2\lambda})\tau\|\eta_{h}^{n}\|^{2} 
			\\
			&\ + (\frac{c_{0}}{2}+\frac{\alpha^{2}}{2\lambda})\tau\|e_{3}(t_{n})\|^{2} 
			+ (\frac{c_{0}}{2}+\frac{\alpha^{2}}{2\lambda})\tau\|\eta_{h}^{n}\|^{2} 
			+ \frac{1}{2\lambda}\tau\|\bar\partial_{t}\pi(t_{n})\|^{2} 
			\\
			&\ + \frac{\alpha^{2}}{2\lambda}\tau\|\eta_{h}^{n}\|^{2} 
			+ \frac{1}{2\lambda}\tau\|e_{5}(t_{n})\|^{2} 
			+ \frac{\alpha^{2}}{2\lambda}\tau\|\eta_{h}^{n}\|^{2} 
			\\
			&\ + \frac{L^{2}}{2\kappa}\tau\|\Xi(t_{n})\|^{2} 
			+ \frac{\kappa}{2}\tau\|\nabla \eta_{h}^{n}\|^{2}.
		\end{aligned}
	\end{equation}
	By virtue of \eqref{e7.4.9} and \eqref{e7.4.18}, we have
	\begin{equation}\label{e7.4.24}
		\begin{aligned}
			\delta_{0}\|\zeta_{h}^{n}\| \leq&\ \sup_{\mathbf{0}\neq \mathbf{v}_{h}\in \mathbb{U}_{h}}\frac{\bar{b}(\mathbf{v}_{h},\zeta_{h}^{n})}{\|\mathbf{v}_{h}\|_{1}} 
			= \sup_{\mathbf{0}\neq \mathbf{v}_{h}\in \mathbb{U}_{h}}\frac{\bar{a}(\tilde{\boldsymbol{\xi}}_{h}^{n}, \mathbf{v}_{h})}{\|\mathbf{v}_{h}\|_{1}} 
			\\
			\leq&\ \sup_{\mathbf{0}\neq \mathbf{v}_{h}\in \mathbb{U}_{h}}\frac{C\|\tilde{\boldsymbol{\xi}}_{h}^{n}\|_{1}\|\mathbf{v}_{h}\|_{1}}{\|\mathbf{v}_{h}\|_{1}}
			= C\|\tilde{\boldsymbol{\xi}}_{h}^{n}\|_{1},
		\end{aligned}
	\end{equation}
	where $C>0$ is a constant independent of $\tau$, $h$ and $\lambda$. 
	Inductively for $n\geq 1$ in \eqref{e7.4.23}, and then utilizing \eqref{e7.4.24}, we conclude that for $1\leq n\leq N$,
	\begin{equation}\label{e7.4.25}
		\begin{aligned}
			&\ \epsilon\|\Pi_{h}^{n}\|^{2} 
			+ \mu\|\Theta_{h}^{n}\|^{2} 
			+ c_{1}^{2}\|\tilde{\boldsymbol{\xi}}_{h}^{n}\|_{1}^{2}
			+ c_{0}\|\eta_{h}^{n}\|^{2} 
			\\
			\leq&\ \epsilon\|\Pi_{h}^{n}\|^{2} 
			+ \mu\|\Theta_{h}^{n}\|^{2} 
			+ \|\tilde{\boldsymbol{\xi}}_{h}^{n}\|_{\bar{a}}^{2}
			+ c_{0}\|\eta_{h}^{n}\|^{2} 
			+ \frac{1}{\lambda}\|\alpha\eta_{h}^{n}-\zeta_{h}^{n}\|^{2} 
			\\
			\leq&\ \epsilon\|\Pi_{h}^{0}\|^{2} 
			+ \mu\|\Theta_{h}^{0}\|^{2} 
			+ \|\tilde{\boldsymbol{\xi}}_{h}^{0}\|_{\bar{a}}^{2}
			+ c_{0}\|\eta_{h}^{0}\|^{2} 
			+ \frac{1}{\lambda}\|\alpha\eta_{h}^{0}-\zeta_{h}^{0}\|^{2} 
			\\
			&\ +\epsilon\sum_{i=1}^{n}\tau\|\bar\partial_{t}\Xi(t_{i})\|^{2} 
			+ \epsilon\sum_{i=1}^{n}\tau\|\bm{e}_{1}(t_{i})\|^{2} 
			+ (\sigma+\frac{L^{2}}{\kappa})\sum_{i=1}^{n}\tau\|\Xi(t_{i})\|^{2} 
			\\
			&\ + L\sum_{i=1}^{n}\tau\|\nabla \gamma(t_{i})\|^{2} 
			+ \mu\sum_{i=1}^{n}\tau\|\bar\partial_{t}\Lambda(t_{i})\|^{2} 
			+ \mu\sum_{i=1}^{n}\tau\|\bm{e}_{2}(t_{i})\|^{2} 
			\\
			&\ + \sum_{i=1}^{n}\tau\|\nabla\times \Xi(t_{i})\|^{2} 
			+ (c_{0}+\frac{2\alpha^{2}}{\lambda})\sum_{i=1}^{n}\tau\|\bar\partial_{t}\gamma(t_{i})\|^{2} 
			+ \sum_{i=1}^{n}\tau\|\bar\partial_{t}\tilde{\boldsymbol{\rho}}(t_{i})\|_{1}^{2} 
			\\
			&\ + (c_{0}+\frac{\alpha^{2}}{\lambda})\sum_{i=1}^{n}\tau\|e_{3}(t_{i})\|^{2} 
			+ \frac{1}{\lambda}\sum_{i=1}^{n}\tau\|\bar\partial_{t}\pi(t_{i})\|^{2} 
			+ \frac{1}{\lambda}\sum_{i=1}^{n}\tau\|e_{5}(t_{i})\|^{2} 
			\\
			&\ + (2\epsilon+3\sigma+L+\frac{4L^{2}}{\kappa})\sum_{i=1}^{n}\tau\|\Pi_{h}^{i}\|^{2} 
			+ (2\mu+1)\sum_{i=1}^{n}\tau\|\Theta_{h}^{i}\|^{2} 
			\\
			&\ + (\frac{1}{\lambda}+1)C\sum_{i=1}^{n}\tau\|\tilde{\boldsymbol{\xi}}_{h}^{i}\|_{1}^{2} 
			+ (2c_{0}+\frac{4\alpha^{2}}{\lambda})\sum_{i=1}^{n}\tau\|\eta_{h}^{i}\|^{2}.
		\end{aligned}
	\end{equation}
	Using \eqref{e7.4.11} and Lemmas \ref{l7.3.1} and \ref{l7.4.3} yields
	\begin{equation}\label{e7.4.26}
		\begin{aligned}
			\|\tilde{\boldsymbol{\xi}}_{h}^{0}\|_{\bar{a}} 
			\leq&\ c_{2}\|\mathscr{Q}_{h}\mathbf{u}_{0}-\mathbf{u}_{0}\|_{1} 
			+ c_{2}\|\mathbf{u}_{0}-\mathscr{P}_{h}\mathbf{u}_{0}\|_{1}
			\leq Ch^{\min\{l,\beta\}}\|\mathbf{u}_{0}\|_{\beta+1},
			\\
			\|\zeta_{h}^{0}\| =&\ 0.
		\end{aligned}
	\end{equation}
	In a similar way as proving \eqref{e7.3.21}, utilizing Lemmas \ref{l7.4.2} and \ref{l7.4.3}, we have
	\begin{equation}\label{e7.4.27}
		\begin{aligned}
			\|\bar\partial_{t}\tilde{\boldsymbol{\rho}}(t_{i})\|_{1}^{2} \leq&\ C\frac{1}{\tau} h^{2\min\{l,\beta\}} \int_{t_{i-1}}^{t_{i}}{\|\frac{\partial}{\partial t} \mathbf{u}(\theta)\|_{\beta+1}^{2}}{\, \mathrm{d}\theta},
			\\
			\|\bar\partial_{t}\pi(t_{i})\|^{2} \leq&\ C\frac{1}{\tau} h^{2\min\{l,\beta\}} \int_{t_{i-1}}^{t_{i}}{\|\frac{\partial}{\partial t} \varphi(\theta)\|_{\beta}^{2}}{\, \mathrm{d}\theta}.
		\end{aligned}
	\end{equation}
	In a similar way as deriving \eqref{e7.3.23}, we get
	\begin{equation}\label{e7.4.28}
		\|e_{5}(t_{i})\|^{2}\leq \tau\int_{t_{i-1}}^{t_{i}}{\|\frac{\partial^{2}}{\partial t^{2}} \varphi(\theta)\|^{2}}{\, \mathrm{d}\theta}.
	\end{equation}
	According to \eqref{e7.4.25}--\eqref{e7.4.28}, \eqref{e7.3.20}--\eqref{e7.3.24}, Lemma \ref{l7.3.2} and (H2)--(H4), we obtain
	\begin{align*}
		&\ \epsilon\|\Pi_{h}^{n}\|^{2} 
		+ \mu\|\Theta_{h}^{n}\|^{2} 
		+ c_{1}^{2}\|\tilde{\boldsymbol{\xi}}_{h}^{n}\|_{1}^{2}
		+ c_{0}\|\eta_{h}^{n}\|^{2} 
		\\
		\leq&\ \epsilon Ch^{2\min\{l,\beta\}}\|\mathbf{E}_{0}\|_{\beta+1}^{2} 
		+ Ch^{2\min\{l,\beta\}}\|\mathbf{u}_{0}\|_{\beta+1}^{2}
		+ c_{0}Ch^{2(\min\{l,\beta\}+1)}\|p_{0}\|_{\beta+1}^{2} 
		\\
		&\ + \frac{2\alpha^{2}}{\lambda}Ch^{2(\min\{l,\beta\}+1)}\|p_{0}\|_{\beta+1}^{2} 
		+\epsilon Ch^{2\min\{l,\beta\}} \int_{0}^{t_{n}}{\|\frac{\partial}{\partial t} \mathbf{E}(\theta)\|_{\beta+1}^{2}}{\, \mathrm{d}\theta} 
		\\
		&\ + \epsilon \tau^{2}\int_{0}^{t_{n}}{\|\frac{\partial^{2}}{\partial t^{2}} \mathbf{E}(\theta)\|^{2}}{\, \mathrm{d}\theta} 
		+ (\sigma+\frac{L^{2}}{\kappa})Ch^{2\min\{l,\beta\}}\sum_{i=1}^{n}\tau\|\mathbf{E}(t_{i})\|_{\beta+1}^{2} 
		\\
		&\ + LCh^{2\min\{l,\beta\}}\sum_{i=1}^{n}\tau\|p(t_{i})\|_{\beta+1}^{2} 
		+ \mu Ch^{2\min\{l,\beta\}} \int_{0}^{t_{n}}{\|\frac{\partial}{\partial t} \mathbf{H}(\theta)\|_{\beta}^{2}}{\, \mathrm{d}\theta} 
		\\
		&\ + \mu \tau^{2}\int_{0}^{t_{n}}{\|\frac{\partial^{2}}{\partial t^{2}} \mathbf{H}(\theta)\|^{2}}{\, \mathrm{d}\theta} 
		+ Ch^{2\min\{l,\beta\}}\sum_{i=1}^{n}\tau\|\mathbf{E}(t_{i})\|_{\beta+1}^{2} 
		\\
		&\ + (c_{0}+\frac{2\alpha^{2}}{\lambda})Ch^{2(\min\{l,\beta\}+1)} \int_{0}^{t_{n}}{\|\frac{\partial}{\partial t} p(\theta)\|_{\beta+1}^{2}}{\, \mathrm{d}\theta} 
		\\
		&\ + Ch^{2\min\{l,\beta\}} \int_{0}^{t_{n}}{\|\frac{\partial}{\partial t} \mathbf{u}(\theta)\|_{\beta+1}^{2}}{\, \mathrm{d}\theta} 
		+ (c_{0}+\frac{\alpha^{2}}{\lambda})\tau^{2}\int_{0}^{t_{n}}{\|\frac{\partial^{2}}{\partial t^{2}} p(\theta)\|^{2}}{\, \mathrm{d}\theta} 
		\\
		&\ + \frac{1}{\lambda}Ch^{2\min\{l,\beta\}} \int_{0}^{t_{n}}{\|\frac{\partial}{\partial t} \varphi(\theta)\|_{\beta}^{2}}{\, \mathrm{d}\theta} 
		+ \frac{1}{\lambda}\tau^{2}\int_{0}^{t_{n}}{\|\frac{\partial^{2}}{\partial t^{2}} \varphi(\theta)\|^{2}}{\, \mathrm{d}\theta} 
		\\
		&\ + (2\epsilon+3\sigma+L+\frac{4L^{2}}{\kappa})\sum_{i=1}^{n}\tau\|\Pi_{h}^{i}\|^{2} 
		+ (2\mu+1)\sum_{i=1}^{n}\tau\|\Theta_{h}^{i}\|^{2} 
		\\
		&\ + (\frac{1}{\lambda}+1)C\sum_{i=1}^{n}\tau\|\tilde{\boldsymbol{\xi}}_{h}^{i}\|_{1}^{2} 
		+ (2c_{0}+\frac{4\alpha^{2}}{\lambda})\sum_{i=1}^{n}\tau\|\eta_{h}^{i}\|^{2}
		\\
		\leq&\ C(\tau^{2}+h^{2\min\{l,\beta\}}) 
		+ C\sum_{i=1}^{n}\tau(\|\Pi_{h}^{i}\|^{2}+\|\Theta_{h}^{i}\|^{2}+\|\tilde{\boldsymbol{\xi}}_{h}^{i}\|_{1}^{2}+\|\eta_{h}^{i}\|^{2}).
	\end{align*}
	Then, applying the discrete Gr\"{o}nwall lemma yields
	\begin{align*}
		\|\Pi_{h}^{n}\|^{2} 
		+ \|\Theta_{h}^{n}\|^{2} 
		+ \|\tilde{\boldsymbol{\xi}}_{h}^{n}\|_{1}^{2}
		+ \|\eta_{h}^{n}\|^{2} 
		\leq&\ C (\tau^{2} + h^{2\min\{l,\beta\}}) e^{C\sum_{i=1}^{n}\tau} 
		\\
		\leq&\ C (\tau^{2} + h^{2\min\{l,\beta\}}),
	\end{align*}
	which together with \eqref{e7.4.24} leads to
	\begin{equation*}
		\|\Pi_{h}^{n}\|^{2} 
		+ \|\Theta_{h}^{n}\|^{2} 
		+ \|\tilde{\boldsymbol{\xi}}_{h}^{n}\|_{1}^{2} 
		+ \|\zeta_{h}^{n}\|^{2} 
		+ \|\eta_{h}^{n}\|^{2} 
		\leq C (\tau^{2} + h^{2\min\{l,\beta\}}),
	\end{equation*}
	where $C>0$ is a constant independent of $\tau$, $h$ and $\lambda$. This together with \eqref{e7.3.12} and \eqref{e7.4.17} completes the proof.
\end{proof}

\section{Numerical experiments}\label{s7.5}
Fix $T=0.1$, $\mathscr{D}=[0, 1]^{3}$, $\epsilon=1$, $\sigma=2$, $L=1$, $\mu=1$, $G=1$, $\alpha=1$, $c_{0}=1$ and $\kappa=2$. We take uniform tetrahedral partition for $\mathscr{D}$ to execute numerical experiments with $l=2$ on the FEniCS computing platform \cite{AlBlHaJoKeLoRiRiRoWe2015,LoMaWe2012}. 

\begin{example}\label{eg7.5.1}
	The initial value functions $\mathbf{E}_{0}$, $\mathbf{H}_{0}$, $\mathbf{u}_{0}$, $p_{0}$ and right-hand side functions $\mathbf{j}$, $\mathbf{f}$, $g$ are chosen such that the exact solution to \eqref{e7.2.1}--\eqref{e7.2.3} reads
	\begin{align*}
		\mathbf{E} =
		\left(
		\begin{array}{c}
			\sin(\pi t)\sin(\pi x)\sin(\pi y)\sin(\pi z)
			\\
			\sin(\pi t)\sin(\pi x)\sin(\pi y)\sin(\pi z)
			\\
			\sin(\pi t)\sin(\pi x)\sin(\pi y)\sin(\pi z)
		\end{array}
		\right),
	\end{align*}
	\begin{align*}
		\mathbf{H} =
		\left(
		\begin{array}{c}
			\cos(\pi t)(\sin(\pi x)\cos(\pi y)\sin(\pi z)-\sin(\pi x)\sin(\pi y)\cos(\pi z))
			\\
			\cos(\pi t)(\sin(\pi x)\sin(\pi y)\cos(\pi z)-\cos(\pi x)\sin(\pi y)\sin(\pi z))
			\\
			\cos(\pi t)(\cos(\pi x)\sin(\pi y)\sin(\pi z)-\sin(\pi x)\cos(\pi y)\sin(\pi z))
		\end{array}
		\right),
	\end{align*}
	\begin{align*}
		\mathbf{u} =
		\left(
		\begin{array}{c}
			200e^{-t}(x-x^{2})^{2}(2y^{3}-3y^{2}+y)(2z^{3}-3z^{2}+z)+\frac{1}{\lambda}e^{-t}\sin(\pi x)
			\\
			-100e^{-t}(y-y^{2})^{2}(2x^{3}-3x^{2}+x)(2z^{3}-3z^{2}+z)+\frac{1}{\lambda}e^{-t}\sin(\pi y)
			\\
			-100e^{-t}(z-z^{2})^{2}(2y^{3}-3y^{2}+y)(2x^{3}-3x^{2}+x)+\frac{1}{\lambda}e^{-t}\sin(\pi z)
		\end{array}
		\right),
	\end{align*}
	\begin{equation*}
		p = e^{-t}\sin(\pi x)\sin(\pi y)\sin(\pi z).
	\end{equation*}
	Here, $\nabla\cdot\mathbf{u} = \frac{1}{\lambda}\pi e^{-t}(\cos(\pi x)+\cos(\pi y)+\cos(\pi z)) \to 0$ as $\lambda\to\infty$.
\end{example}

We take $\tau=1/1800$ and $h= \frac{1}{3},\frac{1}{5},\frac{1}{7},\frac{1}{9}$ to conduct numerical experiments using the conforming FEM \eqref{e7.3.6} for $\lambda=1,10^{4},10^{8}$, respectively, to examine the Poisson locking phenomenon, see Table \ref{tab7.5.1} for the corresponding errors and convergence orders. We observe that the convergence order of $\mathbf{u}$ reduces by $1$, which indicates the locking phenomenon does occur as $\lambda$ increases. Then, we carry out numerical computations with locking-free FEM \eqref{e7.4.12} and list the corresponding errors and convergence orders in Table \ref{tab7.5.2}, which shows the locking-free property.

\begin{table}[htbp]
	\belowrulesep=0pt
	\aboverulesep=0pt
	\renewcommand{\arraystretch}{1.4}
	\setlength{\arrayrulewidth}{0.1mm}
	\centering
	\caption{Errors and convergence orders of conforming FEM \eqref{e7.3.6}.}
	\label{tab7.5.1}
	\begin{tabular}{@{}c|cc|cc|cc@{}}
		\toprule
		$\lambda$ & \multicolumn{2}{c|}{$\lambda=1$} & \multicolumn{2}{c|}{$\lambda=10^{4}$} & \multicolumn{2}{c}{$\lambda=10^{8}$}
		\\
		\midrule[0.7pt]
		$h$ & $\big\|\mathbf{E}(T)-\mathbf{E}_{h}^{N}\big\|$ & Order & $\big\|\mathbf{E}(T)-\mathbf{E}_{h}^{N}\big\|$ & Order & $\big\|\mathbf{E}(T)-\mathbf{E}_{h}^{N}\big\|$ & Order
		\\
		\midrule[0.2pt]
		1/3 & 8.5384E-02 & - & 8.5382E-02 & -  & 8.5382E-02 & -
		\\
		1/5 & 5.1133E-02 & 1.0037 & 5.1133E-02 & 1.0037 & 5.1133E-02 & 1.0037
		\\
		1/7 & 3.6551E-02 & 0.9977 & 3.6551E-02 & 0.9978 & 3.6551E-02 & 0.9978
		\\
		1/9 & 2.8477E-02 & 0.9933 & 2.8477E-02 & 0.9933 & 2.8477E-02 & 0.9933
		\\
		\midrule[0.7pt]
		$h$ & $\big\|\mathbf{H}(T)-\mathbf{H}_{h}^{N}\big\|$ & Order & $\big\|\mathbf{H}(T)-\mathbf{H}_{h}^{N}\big\|$ & Order & $\big\|\mathbf{H}(T)-\mathbf{H}_{h}^{N}\big\|$ & Order
		\\
		\midrule[0.2pt]
		1/3 & 8.5011E-02 & - & 8.5011E-02 & -  & 8.5011E-02 & -
		\\
		1/5 & 3.0650E-02 & 1.9970 & 3.0650E-02 & 1.9970 & 3.0650E-02 & 1.9970
		\\
		1/7 & 1.5572E-02 & 2.0127 & 1.5572E-02 & 2.0127 & 1.5572E-02 & 2.0127
		\\
		1/9 & 9.4039E-03 & 2.0067 & 9.4039E-03 & 2.0067 & 9.4039E-03 & 2.0067
		\\
		\midrule[0.7pt]
		$h$ & $\big\|\mathbf{u}(T)-\mathbf{u}_{h}^{N}\big\|_{1}$ & Order & $\big\|\mathbf{u}(T)-\mathbf{u}_{h}^{N}\big\|_{1}$ & Order & $\big\|\mathbf{u}(T)-\mathbf{u}_{h}^{N}\big\|_{1}$ & Order
		\\
		\midrule[0.2pt]
		1/3 & 2.2743E-01 & - & 3.6909E-01 & -  & 3.7058E-01 & -
		\\
		1/5 & 9.2242E-02 & 1.7666 & 2.4277E-01 & 0.8201 & 2.4526E-01 & 0.8080
		\\
		1/7 & 4.9157E-02 & 1.8706 & 1.7118E-01 & 1.0385 & 1.7455E-01 & 1.0108
		\\
		1/9 & 3.0367E-02 & 1.9165 & 1.3027E-01 & 1.0867 & 1.3443E-01 & 1.0393
		\\
		\midrule[0.7pt]
		$h$ & $\big\|p(T)-p_{h}^{N}\big\|$ & Order & $\big\|p(T)-p_{h}^{N}\big\|$ & Order & $\big\|p(T)-p_{h}^{N}\big\|$ & Order
		\\
		\midrule[0.2pt]
		1/3 & 1.4784E-02 & - & 1.4790E-02 & -  & 1.4790E-02 & -
		\\
		1/5 & 2.9408E-03 & 3.1613 & 2.9417E-03 & 3.1615 & 2.9417E-03 & 3.1615
		\\
		1/7 & 1.0205E-03 & 3.1457 & 1.0206E-03 & 3.1460 & 1.0206E-03 & 3.1460
		\\
		1/9 & 4.6761E-04 & 3.1052 & 4.6763E-04 & 3.1057 & 4.6763E-04 & 3.1057
		\\
		\bottomrule
	\end{tabular}
\end{table}

\begin{table}[htbp]
	\belowrulesep=0pt
	\aboverulesep=0pt
	\renewcommand{\arraystretch}{1.4}
	\setlength{\arrayrulewidth}{0.1mm}
	\centering
	\caption{Errors and convergence orders of locking-free FEM \eqref{e7.4.12}.}
	\label{tab7.5.2}
	\begin{tabular}{@{}c|cc|cc|cc@{}}
		\toprule
		$\lambda$ & \multicolumn{2}{c|}{$\lambda=1$} & \multicolumn{2}{c|}{$\lambda=10^{4}$} & \multicolumn{2}{c}{$\lambda=10^{8}$}
		\\
		\midrule[0.7pt]
		$h$ & $\big\|\mathbf{E}(T)-\mathbf{E}_{h}^{N}\big\|$ & Order & $\big\|\mathbf{E}(T)-\mathbf{E}_{h}^{N}\big\|$ & Order & $\big\|\mathbf{E}(T)-\mathbf{E}_{h}^{N}\big\|$ & Order
		\\
		\midrule[0.2pt]
		1/3 & 8.5379E-02 & - & 8.5381E-02 & -  & 8.5381E-02 & -
		\\
		1/5 & 5.1133E-02 & 1.0036 & 5.1133E-02 & 1.0037 & 5.1133E-02 & 1.0037
		\\
		1/7 & 3.6551E-02 & 0.9977 & 3.6551E-02 & 0.9977 & 3.6551E-02 & 0.9977
		\\
		1/9 & 2.8477E-02 & 0.9933 & 2.8477E-02 & 0.9933 & 2.8477E-02 & 0.9933
		\\
		\midrule[0.7pt]
		$h$ & $\big\|\mathbf{H}(T)-\mathbf{H}_{h}^{N}\big\|$ & Order & $\big\|\mathbf{H}(T)-\mathbf{H}_{h}^{N}\big\|$ & Order & $\big\|\mathbf{H}(T)-\mathbf{H}_{h}^{N}\big\|$ & Order
		\\
		\midrule[0.2pt]
		1/3 & 8.5011E-02 & - & 8.5011E-02 & -  & 8.5011E-02 & -
		\\
		1/5 & 3.0650E-02 & 1.9970 & 3.0650E-02 & 1.9970 & 3.0650E-02 & 1.9970
		\\
		1/7 & 1.5572E-02 & 2.0127 & 1.5572E-02 & 2.0127 & 1.5572E-02 & 2.0127
		\\
		1/9 & 9.4039E-03 & 2.0067 & 9.4039E-03 & 2.0067 & 9.4039E-03 & 2.0067
		\\
		\midrule[0.7pt]
		$h$ & $\big\|\mathbf{u}(T)-\mathbf{u}_{h}^{N}\big\|_{1}$ & Order & $\big\|\mathbf{u}(T)-\mathbf{u}_{h}^{N}\big\|_{1}$ & Order & $\big\|\mathbf{u}(T)-\mathbf{u}_{h}^{N}\big\|_{1}$ & Order
		\\
		\midrule[0.2pt]
		1/3 & 2.2876E-01 & - & 1.8360E-01 & -  & 1.8361E-01 & -
		\\
		1/5 & 9.2215E-02 & 1.7786 & 7.7342E-02 & 1.6924 & 7.7343E-02 & 1.6925
		\\
		1/7 & 4.9083E-02 & 1.8742 & 4.1717E-02 & 1.8347 & 4.1717E-02 & 1.8347
		\\
		1/9 & 3.0316E-02 & 1.9172 & 2.5927E-02 & 1.8926 & 2.5927E-02 & 1.8926
		\\
		\midrule[0.7pt]
		$h$ & $\big\|\varphi(T)-\varphi_{h}^{N}\big\|$ & Order & $\big\|\varphi(T)-\varphi_{h}^{N}\big\|$ & Order & $\big\|\varphi(T)-\varphi_{h}^{N}\big\|$ & Order 
		\\
		\midrule[0.2pt]
		1/3 & 1.6763E-01 & - & 2.1965E-01 & - & 2.1983E-01 & -
		\\
		1/5 & 5.6581E-02 & 2.1261 & 6.4841E-02 & 2.3885 & 6.4861E-02 & 2.3895
		\\
		1/7 & 2.8187E-02 & 2.0709 & 3.0517E-02 & 2.2398 & 3.0522E-02 & 2.2403
		\\
		1/9 & 1.6872E-02 & 2.0421 & 1.7759E-02 & 2.1543 & 1.7760E-02 & 2.1546
		\\
		\midrule[0.7pt]
		$h$ & $\big\|p(T)-p_{h}^{N}\big\|$ & Order & $\big\|p(T)-p_{h}^{N}\big\|$ & Order & $\big\|p(T)-p_{h}^{N}\big\|$ & Order
		\\
		\midrule[0.2pt]
		1/3 & 1.4686E-02 & - & 1.4790E-02 & -  & 1.4790E-02 & -
		\\
		1/5 & 2.9289E-03 & 3.1562 & 2.9417E-03 & 3.1615 & 2.9417E-03 & 3.1615
		\\
		1/7 & 1.0181E-03 & 3.1404 & 1.0206E-03 & 3.1460 & 1.0206E-03 & 3.1460
		\\
		1/9 & 4.6693E-04 & 3.1018 & 4.6763E-04 & 3.1057 & 4.6763E-04 & 3.1057
		\\
		\bottomrule
	\end{tabular}
\end{table}

\section*{Acknowledgements}
This work is supported by National Natural Science Foundation of China (12571428).

\section*{Conflict of interest}
The authors declare that there is no conflict of interest.

\section*{Data Availability Statement}
No data set is used in the research.

	
	
	
	
	
%
%
%

\bibliographystyle{elsarticle-num-names}
\bibliography{High-order_Locking-free_Electroporoelasticity_references}

\end{document}